\documentclass[final,onefignum,onetabnum]{siamart171218}
\usepackage{mathptmx}
\usepackage{amsfonts}
\usepackage{amsmath}
\usepackage{amssymb}
\usepackage[most]{tcolorbox}

\usepackage{hyperref}
\usepackage[a4paper, total={6.5in, 8.2in}]{geometry}
\usepackage{xparse}
\usepackage{hhline}
\usepackage[textfont={small, it}, labelfont=bf, skip=5pt]{caption}
\usepackage{soul}
\usepackage[author={Yous van Halder}]{pdfcomment}
\providecommand{\mbf}[1]{\mathbf{#1}}

\providecommand{\der}[2]{\frac{\partial #1}{\partial #2}}

\providecommand{\vecbrace}[1]{\left(\begin{array}{c} #1 \end{array}\right)}
\begin{document}

\title{PDE/PDF-informed adaptive sampling for efficient non-intrusive surrogate modelling}
\author{Y. van Halder \thanks{Centrum Wiskunde \& Informatica (CWI), Science Park 123, 1098 XG, Amsterdam, the Netherlands} \and B. Sanderse \footnotemark[2] \and B. Koren \thanks{Eindhoven University of Technology, P.O. Box 513, 5600 MB, Eindhoven, the Netherlands}}

\maketitle

\begin{abstract}
A novel refinement measure for non-intrusive surrogate modelling of partial differential equations (PDEs) with uncertain parameters is proposed. Our approach uses an empirical interpolation procedure, where the proposed refinement measure is based on a PDE residual and probability density function of the uncertain parameters, and excludes parts of the PDE solution that are not used to compute the quantity of interest. The PDE residual used in the refinement measure is computed by using all the partial derivatives that enter the PDE separately. The proposed refinement measure is suited for efficient parametric surrogate construction when the underlying PDE is known, even when the parameter space is non-hypercube, and has no restrictions on the type of the discretisation method. Therefore, we are not restricted to conventional discretisation techniques, e.g., finite elements and finite volumes, and the proposed method is shown to be effective when used in combination with recently introduced neural network PDE solvers. We present several numerical examples with increasing complexity that demonstrate accuracy, efficiency and generality of the method. 
\end{abstract}

{
%\titleformat{\section}[display]
%{\bfseries\Large}  
%{\filright{\tiny} %\thesection
%}  
%  {1ex}
%  {\vspace{1ex}\filright}
%  [\vspace{1ex}]
% \linenumbers

\section{Introduction}~\\
\noindent Uncertainty Quantification (UQ) has become increasingly important for complex engineering applications. Determining and quantifying the influence of parametric and  model-form uncertainties is essential for a wide range of applications: from turbulent flow phenomena \cite{xiao_quantifying_2016, edeling_simplex-stochastic_2016}, aerodynamics \cite{simon_gpc-based_2010}, biology \cite{cho_experimentaldesign_2003, abagyan_biased_1994} to design optimisation \cite{constantinescu_computational_2011, mateos_monte_2000, papadrakakis_reliability-based_2002}.

Numerical methods in UQ are often divided in two groups; intrusive and non-intrusive. We focus on non-intrusive sampling methods, as they do not change the deterministic model and allow for the usage of black-box solvers. One commonly used non-intrusive method is \textit{stochastic collocation} \cite{xiu_numerical_2010}, which uses a black-box to sample the deterministic model several times in stochastic space, and interpolates these samples to construct a surrogate model. Commonly used sets of interpolation nodes are the Gauss nodes \cite{xiu_numerical_2010} and Clenshaw-Curtis nodes \cite{boyd_chebyshev_2013}. The Gauss nodes possess a high polynomial exactness, but are not nested, which makes them less attractive when the surrogate needs to be refined. On the other hand, Clenshaw-Curtis nodes are nested and are suited for accurate surrogate modelling. However, the number of samples increases exponentially with the number of dimensions of the stochastic space, i.e., the number of uncertainties in the model. This phenomenon is known as the \textit{curse of dimensionality} and limits the applicability of stochastic collocation when the black-box is computationally expensive to sample from. As a remedy, alternatives to the tensor based stochastic collocation were introduced, e.g., Leja-node stochastic collocation \cite{narayan_adaptive_2014,laurentleja}, empirical interpolation \cite{BARRAULT2004667,hesthaven_empirical_2016,hesthaven_efficient_2014,ohlberger_reduced_2015}, `best' interpolation \cite{best_interpolation}, and Smolyak sparse grids \cite{nobile_sparse_2008, nobile_anisotropic_2008}. These alternatives sample the model adaptively in order to reduce the number of samples. Furthermore, empirical interpolation enhances adaptive sampling placement by incorporating knowledge from the underlying model in terms of the Partial Differential Equation (PDE) residual, which is a measure of how well an approximation satisfies the model. Even though this results in a significant decrease in the number of samples when compared to methods that do not take the model into account, the method is still intractable for uncertainty propagation with a large number of uncertainties, as the empirical interpolation bases sample placement on the entire solution, rather than focusing on the Quantity of Interest (QoI). Furthermore, when interested in the statistical properties of the QoI, e.g., mean and variance, using only the residual as a measure for adaptive sampling placement is not efficient, as it does not utilise the Probability Density Function (PDF), which is used in the calculation of these quantities. 

In this work, an empirical interpolation procedure related to \cite{hesthaven_empirical_2016,ohlberger_reduced_2015,best_interpolation} is proposed, with the main differences that: probability information is included in the sampling algorithm, fewer restrictions on the type of the PDE are imposed, and the residual is based solely on the QoI and not on the entire PDE solution. A relation between the error in the surrogate and the PDE residual is given, which justifies the residual as a refinement measure for surrogate construction. Furthermore an alternative refinement measure, which incorporates the PDF, is proposed when interested in the statistical properties of the QoI. Using both the residual and PDF as a measure for defining new sampling locations, leads to accurate statistical properties of the QoI, which converge significantly faster than the procedures in \cite{hesthaven_empirical_2016,ohlberger_reduced_2015,best_interpolation}. Additionaly, it is shown that the proposed method is suited for efficient surrogate construction on complex topologies, which is a common problem in the case of dependent input uncertainties. Our method does not require a specific type of PDE discretisation, e.g., finite elements or finite volumes, and can therefore be used in combination with new state-of-the-art neural network solvers \cite{raissi_physics_2017}. A key part of our approach is the use of the PDE residual, which is discussed later in more detail. In order to compute this residual, the black-box solver needs to give not only the solution values, but also derivatives with respect to spatial/temporal coordinates. This introduces a small degree of intrusiveness in the approach, although no changes to the model equations are necessary  and our approach is therefore still referred to as a non-intrusive approach. Finally, new methods for solving PDEs \cite{raissi_physics_2017} can be used in combination with our method without altering the black-box solver. As our approach is still considered to be non-intrusive and uses a combination of the PDE residual and PDF of the uncertain parameters as a refinement measure, we refer to the proposed method as Non-Intrusive PDE/PDF-informed Adaptive Sampling (NIPPAS).

This paper is outlined as follows: section \ref{sec:ProblemDescription} introduces the problem, section \ref{sec:Method} introduces the new method and proves that the proposed sampling procedure is suitable for accurate surrogate construction. After introducing the proposed method, section 4 shows the individual steps of the method in more detail. Implementation details are discussed in section 5, and section \ref{sec:Results} demonstrates efficiency and accuracy of our method when applied to several test-cases and compares the results with sparse grid interpolation and classical empirical interpolation.
\section{Problem Description}~\\
\label{sec:ProblemDescription}
\noindent Quantifying the effects of parametric uncertainties in computational engineering typically is a three-step process \cite{chantrasmi_pade-legendre_2011}; the input uncertainties are characterised in terms of a Probability Density Function (PDF); the uncertainties are propagated through the model; and the outputs are post-processed, where the Quantity of Interest (QoI) is expressed in terms of its statistical properties. In the present work we focus on the propagation step, and the input distributions are assumed to be given. The underlying model is a PDE, which is assumed to be of the form
\begin{equation}
\sum_{l=1}^{n_l} g_{l}(\mbf{z}, \mbf{x})L^e_l(v^e;\mbf{x})=S(\mbf{z}, \mbf{x})\ ,\ \ (\mbf{x},\mbf{z})\in D\times I_{\mbf{z}}\ ,
\label{eqPD:ModelProblemexact}
\end{equation}
which is supplemented with proper initial and boundary conditions. In \eqref{eqPD:ModelProblemexact}, $n_l$ denotes the number of differential operators in the PDE, $\mbf{z}\in I_{\mbf{z}}\subset \mathbb{R}^{d}$ a $d$-dimensional vector containing uncertain parameters with corresponding joint PDF $\rho(\mbf{z})$, $\mbf{x}\in D$ a vector containing spatial and/or temporal coordinates, $L_l^e$ are differential operators, $g_l$ are known functions, $S$ a source term, and $v^e(\mbf{z}, \mbf{x})$ the exact solution of the PDE. This particular PDE-form assumes that the uncertainties enter the equation via the source term $S(\mbf{z}, \mbf{x})$ and the functions $g_l(\mbf{z}, \mbf{x})$ as parameters in front of differential operators and allows the definition of a non-zero residual in the random space $I_\mbf{z}$. This PDE-form does not comprise all possible PDEs, but many PDEs with parametric uncertainties, e.g., isotropic diffusion equations, Newtonian Navier-Stokes equations and advection-diffusion equations, may be rewritten in this form:
\begin{align}
v^e_t &= z \Delta v^e\ ,\\
v^e_t + (v^e\cdot\nabla)v^e &= -\frac{\nabla p^e}{\rho} + z \Delta v^e\ ,\label{eq:NS}\\
v^e_t + z_1 v^e_x &= z_2 v^e_{xx}\ .
\end{align}

Because the exact solution of \eqref{eqPD:ModelProblemexact} is often not available, a discrete solution vector $\mbf{v}(\mbf{z})$ is computed, e.g., via a finite-difference method or finite-volume method, which satisfies a discretised form of \eqref{eqPD:ModelProblemexact} for $\mbf{z}\in I_{\mbf{z}}$:
\begin{equation}
\sum_{l=1}^{n_l} G_l(\mbf{z}, X)L_l(\mbf{v}(\mbf{z})) = \mbf{S}(\mbf{z}, X)\ ,\ \ X\subset D, \mbf{v}\in \mathbb{R}^{N_{PDE}}\ ,
\label{eqPD:ModelProblem}
\end{equation}
where $L_l:\mathbb{R}^{N_{\text{PDE}}}\rightarrow \mathbb{R}^{N_{\text{PDE}}}$ are the discretised PDE operators, $X=(\mbf{x}_1, ..., \mbf{x}_{N_{\text{PDE}}})$ is the computational grid in space and time consisting of $N_{PDE}$ grid points, $G_l\in\mathbb{R}^{N_{\text{PDE}}\times N_{\text{PDE}}}$ are diagonal matrices with diagonal entries $G_{l, ii}(\mbf{z}, X) = g_l(\mbf{z}, X_i)$, and $\mbf{S}(\mbf{z}, X)\in\mathbb{R}^{N_{\text{PDE}}}$ is the source term evaluated on the computational grid. The initial and boundary conditions are comprised in the source term $\mbf{S}$. The uncertainties enter the equations via the source term and the matrices $G_l$, which comprises the function $g_l$ evaluated on the computational grid $X$. 

We are interested in the dependence of the QoI on the parameters $\mbf{z}$. The QoI $\mbf{u}(\mbf{z})$ is assumed to be a set of linear combinations of the solution vector $\mbf{v}$, i.e., $\mbf{u}(\mbf{z}) = Q\mbf{v}(\mbf{z})$, where
\begin{equation}
Q:\mathbb{R}^{N_{\text{PDE}}} \rightarrow \mathbb{R}^{N_{\text{QoI}}}\ ,
\end{equation}
is a matrix that maps the solution vector $\mbf{v}$ to the QoI $\mbf{u}$. This assumption allows for a suitable refinement measure for sampling, which is introduced later. By assuming linearity of $Q$ we limit the space of possible QoIs, but this limitation is not too severe as many quantities, e.g., integral quantities and mean quantities, can be written in this form.

The goal in this work is twofold: either construct an accurate surrogate for $\mbf{u}(\mbf{z})$ or calculate statistical properties of the QoI with respect to the uncertain parameters $\mbf{z}$. Calculating statistical properties is achieved by constructing a surrogate, which is used in combination with Monte-Carlo sampling to extract the statistical quantities. However if we are only interested in the statistical properties of the QoI, the surrogate does not need to be accurate everywhere, as some areas of the random space contribute little when calculating these statistical properties. Nevertheless, whether we are interested in constructing an accurate surrogate to study the dependency of the QoI on the parameters $\mbf{z}$, or whether we are interested in the statistical properties of the QoI, a surrogate needs to be constructed. We construct a surrogate by applying a PDE-solver to \eqref{eqPD:ModelProblem} and by sampling values from the unknown $\mbf{u}(\mbf{z})$. We sample the QoI $\mbf{u}(\mbf{z}_i)$ at $N+1$ locations $\lbrace\mbf{z}_i\rbrace_{i=0}^N$ in the random space $I_{\mbf{z}}$. The QoI evaluations $\mbf{u}(\mbf{z}_i)$ are calculated as follows
\begin{equation}
\underbrace{\sum_{l=1}^{n_l} G_l(\mbf{z}, X)L_l(\mbf{v}(\mbf{z}))=\mbf{S}(\mbf{z}, X)}_{\text{solve PDE for }\mbf{z}=\mbf{z}_i} \Rightarrow\ \mbf{u}(\mbf{z}_i) = Q\mbf{v}(\mbf{z}_i)\ .
\label{eq:Samples}
\end{equation}
The PDE-solver is assumed to be a black-box, which means that we supply inputs and receive outputs, without the possibility to observe intermediate steps.
After sampling, a surrogate model $\tilde{\mbf{u}}(\mbf{z})$ is constructed by means of polynomial interpolation on the samples $\{(\mbf{z}_i, \mbf{u}(\mbf{z}_i))\}_{i=0}^N$. The interpolant is constructed individually for each element of $\mbf{u}$, such that:
\begin{equation}
\tilde{u}_j(\mbf{z})\approx u_j(\mbf{z}),\ \ \text{for all}\ \ \mbf{z}\in I_{\mbf{z}}\ ,\ \ \text{with}\ \tilde{u}_j(\mbf{z}_i) = u_j(\mbf{z}_i)\ \ i=0,...,N\ ,
\label{eq:surrogatedef}
\end{equation}
where $u_j$ corresponds to the $j$-th element of the vector $\mbf{u}$. The element-wise approximation for the entire QoI vector $\mbf{u}$ is denoted as $\tilde{\mbf{u}}$. Polynomial interpolation is used instead of polynomial regression as it allows for more efficient adaptive refinement and has extensive theoretical grounding. When interpolating, placing a new sample ensures that the new surrogate model is accurate in a neighbourhood around the newly added sample and has therefore immediate impact on the surrogate in the sampled area. This ensures improved accuracy near the new sample location, something which is not necessarily the case when using regression. Choosing proper sample locations $\mbf{z}_i$ is crucial for stable and accurate interpolation and this will be the main focus of this paper. 

Many UQ methods focus either on the PDF \cite{xiu_numerical_2010} or on the PDE residual \cite{hesthaven_empirical_2016} for adaptive sample placement. The PDF indicates which values for $\mbf{z}$ are likely to happen and are therefore important to sample. The PDE residual however gives an indication where surrogate refinement is needed in order for the surrogate to satisfy the underlying PDE. In this paper we propose a novel strategy, in which both the importance of the PDE and PDF is taken into account in determining the sample locations. Finding a set of sample locations, which resembles the importance of both the underlying PDE and the PDF, is the goal of this paper.

\section{Non-Intrusive PDE/PDF-informed Adaptive Sampling}~\\
\label{sec:Method}
\noindent The locations of the interpolation samples determine the accuracy and stability of the surrogate model. We construct a set of interpolation nodes adaptively, by using knowledge from the underlying PDE as refinement measure.
\subsection*{Residual definition}~\\
\noindent For this purpose we define the PDE residual, which indicates how well an approximate solution satisfies the discretised PDE in the random space $I_{\mbf{z}}$. An intuitive definition of the residual can be obtained by first constructing approximations $\widetilde{L_l(\mbf{v}(\mbf{z}))}$ in the random space based on evaluations $L_l(\mbf{v}(\mbf{z}_i))$, and substituting these approximations in the discretised PDE \eqref{eqPD:ModelProblem}:
\begin{equation}
\mbf{R}_v(\mbf{z}):=\sum_{l=1}^{n_l} G_l(\mbf{z}, X)\widetilde{L_l(\mbf{v}(\mbf{z}))} - \mbf{S}(\mbf{z}, X)\ .
\label{eq:PreResidual}
\end{equation}
The quantity $\mbf{R}_v$ indicates how well the approximations $\widetilde{L_l(\mbf{v}(\mbf{z}))}$ satisfy the discretised PDE in the random space $I_{\mbf{z}}$.

\subsection*{Relation between residual and surrogate error}~\\
\noindent The residual $\mbf{R}_v(\mbf{z})$ is a quantity that indicates the quality of a surrogate by substituting the surrogate into the discretised PDE and by calculating the error. The next theorem states a relation between the residual $\mbf{R}_v$ and the error in the surrogate $\tilde{\mbf{u}}$ for linear PDEs, which is used later to define a suitable refinement measure for placing new samples in the random space.
\begin{theorem}
\label{thm:RtoSurr}
Assume the following:
\begin{itemize}
\item Bounded 1D random space $I_z = [z_{\hbox{lb}}, z_{\hbox{rb}}]$.
\item Well-posed discretised linear PDE of the form
\begin{equation}
\sum_{l=1}^{n_l} G_l(z, X) L_l(\mbf{v}(z))=\mbf{S}(z, X)\ ,\ \ X\subset D, \mbf{v}\in \mathbb{R}^{N_{PDE}}\ .
\label{eqthm1:equation}
\end{equation}
\item $L_l$ are discretised linear differential operators, i.e., matrices satisfying:
\begin{equation}
L_l(\mbf{v} + \mbf{w}) = L_l\mbf{v} + L_l\mbf{w}\ .
\end{equation}
\item $G_l(z, X)$ are known matrices as functions of $z$ and $X$.
\item QoI vector can be written as $\mbf{u} = Q\mbf{v}$, where $Q\in\mathbb{R}^{N_{\text{QoI}}\times N_{\text{PDE}}}$.
\end{itemize}
Then the following holds:
\begin{equation}
Q\left(\sum_{l=1}^{n_l} G_l(z, X)L_l\right)^{-1}(\mbf{R}_v(z)) = \tilde{\mbf{u}}(z) - \mbf{u}(z)\ ,
\label{eqthm1:R}
\end{equation}
where $\tilde{\mbf{u}}$ is the interpolant in the random space $I_z$ for each element of the vector $\mbf{u}$, based on the samples $\{(z_i, \mbf{u}(z_i)\}_{i=0}^N$.
\end{theorem}
\begin{proof}
Interpolation in 1D is unique and for convenience we choose the Lagrange basis $\ell_i(z)$, where $\ell_i$ is the $i$-th Lagrange basis polynomial defined as:
\begin{equation}
\ell_i(z) = \prod_{\stackrel{k=0}{k\neq i}}^N \frac{z - z_k}{z_i - z_k}\ .
\end{equation}
The approximations $\widetilde{L_l(\mbf{v}(z))}$ are then given by
\begin{equation}
\widetilde{L_l (\mbf{v}(z))} = \sum_{i=0}^N \ell_i(z) L_l (\mbf{v}(z_i))\ ,
\label{eqthm1:1}
\end{equation}
which can be rewritten due to the linearity of $L$ as
\begin{equation}
\widetilde{L_l (\mbf{v}(z))} = L_l(\sum_{i=0}^N \ell_i(z) \mbf{v}(z_i)) = L_l (\tilde{\mbf{v}}(z))\ .
\end{equation}
Substitution into \eqref{eq:PreResidual} gives
\begin{equation}
\mbf{R}_v(z) = \sum_{l=1}^{n_l} G_l(z, X) L_l(\tilde{\mbf{v}}(z)) - \mbf{S}(\mbf{z}, X)\ .
\end{equation}
Subtracting equation \eqref{eqthm1:equation}, results in
\begin{equation}
\mbf{R}_v(z) =\sum_{l=1}^{n_l} G_l(z, X) L_l(\tilde{\mbf{v}}(z)) - \underbrace{\sum_{l=1}^{n_l} G_l(z, X) L_l(\mbf{v}(z))}_{\mbf{S}(\mbf{z}, X)}\ .
\end{equation}
This can be rewritten by using linearity of $L$ as
\begin{equation}
\mbf{R}_v(z) = \sum_{l=1}^{n_l} G_l(z, X) L_l(\tilde{\mbf{v}}(z) - \mbf{v}(z))\ .
\end{equation}
Using well-posedness of the underlying discretised PDE, we can write 
\begin{equation}
\left(\sum_{l=1}^{n_l} G_l(z, X)L_l\right)^{-1}(\mbf{R}_v(z)) = \tilde{\mbf{v}}(z) - \mbf{v}(z)\ ,
\end{equation}
which can be multiplied by the matrix $Q$ to obtain the desired result
\begin{equation}
Q\left(\sum_{l=1}^{n_l} G_l(z, X)L_l\right)^{-1}(\mbf{R}_v(z)) = \tilde{\mbf{u}}(z) - \mbf{u}(z)\ ,
\end{equation}
$\qed$
\end{proof}

\noindent Theorem \ref{thm:RtoSurr} states a relation between the residual and the error in the QoI surrogate. From this relation we can derive the error bound stated in the following corollary. 
\begin{corollary}
Assume the following:
\begin{itemize}
\item The conditions of theorem 1 are satisfied, such that \eqref{eqthm1:R} holds.
\end{itemize} Then the error $\|\tilde{\mbf{u}}(z) - \mbf{u}(z)\|_2$ satisfies
\begin{equation}
\|\tilde{\mbf{u}}(z) - \mbf{u}(z)\|_2 \leq \|Q\|_2\left\|\left(\sum_{l=1}^{n_l} G_l(z, X)L_l\right)^{-1}\right\|_2\|\mbf{R}_v(z)\|_2\ ,
\label{eqcor:errorbound}
\end{equation}
where $\|\cdot\|_2$ is the vector 2-norm or its induced matrix norm.
\end{corollary}
Corollary 3.1 states an upper bound for the error in the surrogate in terms of the residual and the discretised differential operator, which gives an indication where the surrogate $\tilde{\mbf{u}}$ deviates significantly from the exact solution. Equation \eqref{eqcor:errorbound} holds for any induced matrix norm. In this paper we adopt the 2-norm.

\subsection*{Refinement measure for adaptive sampling}~\\
\noindent The goal is to sample the QoI in the  random space $I_{\mbf{z}}$ such that we can construct an accurate surrogate model for the QoI by means of polynomial interpolation. Theorem 1 shows that a suitable refinement measure for a linear underlying PDE, is given by 
\begin{equation}
\mbf{R}^*(\mbf{z}) = Q\left(\sum_{l=1}^{n_l} G_l(\mbf{z}, X)L_l\right)^{-1}\mbf{R}_v(\mbf{z})\ ,
\label{eq:RExact}
\end{equation}
If samples are placed such that $\mbf{R}^*$ converges to zero, then the error in the surrogate also converges to zero. Using polynomial interpolation for the approximations $\widetilde{L_l(\mbf{v}(\mbf{z}))}$ in $\mbf{R}_v$ ensures that $\mbf{R}_v$ is close to zero in the neighbourhood of the interpolation samples. The quantity $\left(\sum_{l=1}^{n_l} G_l(\mbf{z}, X)L_l\right)^{-1}$ is a function of $\mbf{z}$, but does not depend on the choice of interpolation samples and therefore acts as a scaling function for the residual $\mbf{R}_v$. This justifies using a greedy approach for placing new samples as follows:
\begin{equation}
\mbf{z}_{\text{new}} = \arg \max_{\mbf{z}\in I_{\mbf{z}}} \|\mbf{R}^*(\mbf{z})\|_2\  .
\end{equation}
The construction of $\widetilde{L_l(\mbf{v}(\mbf{z}))}$ in $\mbf{R}_v$ is a crucial step in the NIPPAS method and is therefore discussed in section 4 in detail. However, when using $\mbf{R}^*$ as a refinement measure, there are two issues:
\begin{itemize}
\item The term $\left(\sum_{l=1}^{n_l} G_l(\mbf{z}, X)L_l\right)^{-1}$ is a-priori unknown and expensive to compute.
\item Equation \eqref{eq:RExact} holds for linear PDEs only.
\end{itemize}
The inverse of the full discretisation matrix is unknown in general and cannot be used for adaptive sample placement, as it is expensive to compute as a function of $\mbf{z}$. However, as this term only acts as a scaling function, omitting this term may result in a compact refinement measure for the QoI, which is then computed as follows:
\begin{equation}
\mbf{R}(\mbf{z}) := Q\mbf{R}_v(\mbf{z})\ .
\label{eq:Residual}
\end{equation}
Notice that in comparison to $\mbf{R}^*$, $\mbf{R}$ can be computed for both linear and non-linear PDEs. Intuitively, we might expect that the refinement measure $\mbf{R}(\mbf{z})$ produces accurate and stable interpolants in combination with the greedy sample placement, because large errors and/or instabilities in the surrogate would lead to large errors in the residual, which then triggers new sample placement due to the greedy approach. A commonly used quantity that is used to indicate the quality of a set of sample locations is the Lebesgue constant \cite{trefethen_approximation_2013}. We note that our proposed greedy approach does not necessarily result in sample locations with an optimal Lebesgue constant, but effectively uses the model information to sample regions of interest, similar to other approaches \cite{narayan_adaptive_2014,hesthaven_empirical_2016,best_interpolation,loukrezis_numerical_2017}.

To summarise, $\mbf{R}$ is our proposed refinement measure, as it is less expensive to compute and more generally applicable when compared to $\mbf{R}^*$. A numerical comparison for both refinement measures is given in section \ref{sec:Results} for a case where $\mbf{R}^*$ can still be computed. Furthermore, the effectiveness of using $\mbf{R}$ as a refinement measure for the case of non-linear PDEs is demonstrated numerically in section \ref{sec:Results}.

\subsection*{Incorporating PDF in refinement measure}~\\
\noindent Although using the residual $\mbf{R}(\mbf{z})$ as a refinement measure for adaptive sampling may result in stable and accurate surrogate models, the sampling procedure may put too much resources in areas that contribute little to the statistical quantities. Such areas are of little interest when we compute statistical quantities like the mean
\begin{equation}
\mathbb{E}[\mathbf{u}] :=\int_{I_{\mbf{z}}} \rho(\mbf{z})\mathbf{u}(\mbf{z})\hbox{d}\mbf{z}\ .
\label{eq:meandef}
\end{equation}
These statistical quantities are calculated by weighing the QoI with the PDF and integrating the resulting quantity over the parameter space $I_{\mbf{z}}$. Areas of $I_{\mbf{z}}$ where both the PDF and QoI are low, contribute little to the integral in \eqref{eq:meandef} and therefore an alternative refinement measure is proposed, which is especially suited when one is interested in accurate statistical quantities rather than an accurate surrogate model. The proposed refinement measure is based on the following theorem:

\begin{theorem}
\label{thm:RtoSurrPDF}
Assume the following:
\begin{itemize}
\item The conditions of theorem 1 are satisfied, such that \eqref{eqthm1:R} holds.
\item $\mathbb{E}[\tilde{\mathbf{u}}]$ and $\mathbb{E}[\mathbf{u}]$ are finite.
\end{itemize}
Then the following holds:
\begin{equation}
\|\mathbb{E}[\tilde{\mathbf{u}} - \mathbf{u}]\|_2 \leq (z_{\text{rb}} - z_{\text{lb}})\sup_{z\in[z_{lb}, z_{rb}]}s(z)\rho(z)\|Q\|_2\|\mbf{R}_v(z)\|_2\ ,
\label{eqthm2:R}
\end{equation}
where
\begin{equation}
s(z) = \left\|\left(\sum_{l=1}^{n_l} G_l(z, X)L_l\right)^{-1}\right\|_2\ .
\end{equation}
\end{theorem}
\noindent The proof of the theorem follows the proof of theorem \ref{thm:RtoSurr} and bounds the integral by multiplying the domain length $(z_{\text{rb}} - z_{\text{lb}})$ with a bound for the integrand. The full proof is not shown to avoid repetitiveness.

Theorem \ref{thm:RtoSurrPDF} states a relation between the residual and the error in the mean of the QoI. The theorem shows that a proper refinement measure again includes the inverse of the full differentiation matrix. As stated before, this value is unknown in general as a function of $\mbf{z}$ and is expensive to sample. Therefore the proposed refinement measure for calculating statistical quantities is given by
\begin{equation}
\mbf{R}_{\rho}(\mbf{z}) := \rho(\mbf{z})Q\mbf{R}_v(\mbf{z}) = \rho(\mbf{z})\mbf{R}(\mbf{z})\ ,
\label{eq:ResidualPDF}
\end{equation}
and new samples are placed using the greedy approach
\begin{equation}
\mbf{z}_{\text{new}} = \arg \max_{\mbf{z}\in I_{\mbf{z}}} \|\mbf{R}_{\rho}(\mbf{z})\|_2\  .
\end{equation}
Refinement measure \eqref{eq:ResidualPDF} will not place new samples in parts of the random space where the PDF is zero, but is still able to place samples at ``rare events'' in the random space, i.e., parts where the PDF is small but where the PDE residual is large. 

A comparison of the refinement measures \eqref{eq:Residual} and \eqref{eq:ResidualPDF} for convergence of both surrogate construction and statistical quantity calculation are given in section \ref{sec:Results}.

\subsection*{Overview of method}~\\
\noindent Residual definitions \eqref{eq:Residual} and \eqref{eq:ResidualPDF} are used in the NIPPAS method to adaptively refine a surrogate model $\tilde{\mbf{u}}(\mbf{z})$. In detail, the NIPPAS method comprises the following steps:
\begin{enumerate}
\item initial sample placement
\item refinement loop:
\begin{enumerate}
\item compute residual $\mbf{R}(\mbf{z})$ or $\mbf{R}_{\rho}(\mbf{z})$
\item check stopping criterion
\item find $\mbf{z}_{\text{new}} = \arg \max_{\mbf{z}\in I_{\mbf{z}}} \|\mbf{R}(\mbf{z})\|_2$ or $ \|\mbf{R}_{\rho}(\mbf{z})\|_2$
\item sample model at $\mbf{z}_{\text{new}}$
\end{enumerate}
\item interpolate resulting samples
\end{enumerate}
A schematic representation of the methodology is shown in figure \ref{fig:SchematicMethod}. The choice of using either refinement measure \eqref{eq:Residual} or \eqref{eq:ResidualPDF} depends on whether the interest is in constructing an accurate surrogate or calculating accurate statistical quantities. The individual steps of the proposed surrogate model construction are discussed in detail in the next section. Before discussing the NIPPAS steps, the connection is shown between the NIPPAS method and general empirical interpolation method.

\begin{figure*}[!h]
\centering
\includegraphics[width = \textwidth]{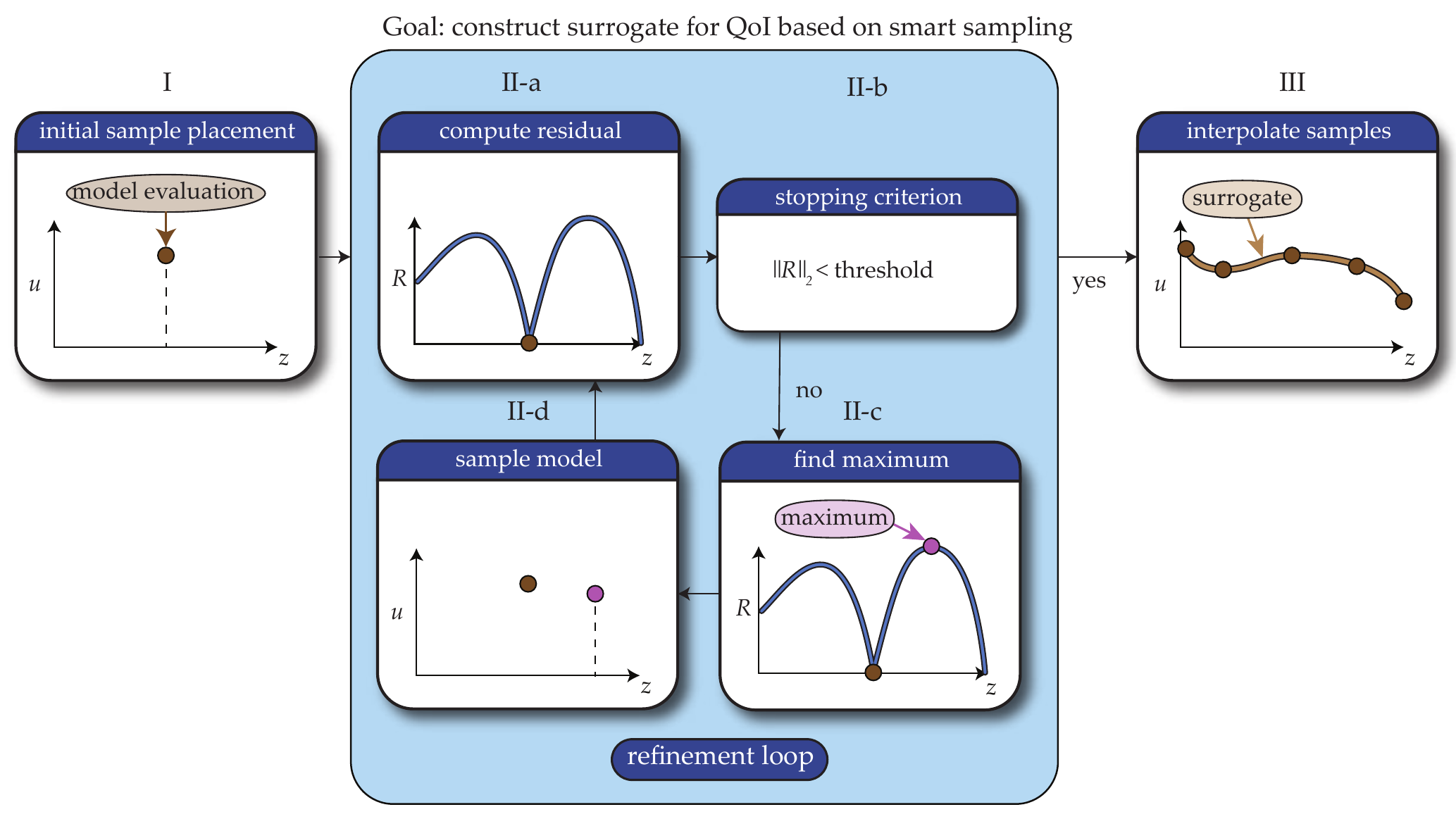}
\caption{\label{fig:SchematicMethod} Schematic representation of the methodology.}
\end{figure*}

The refinement measures were derived from theory, which holds for linear PDEs. We will show in section \ref{sec:Results} that the proposed refinement measures also work for a non-linear underlying PDE. Stable polynomial interpolating surrogates show clustering samples at the boundary of the domain \cite{trefethen_approximation_2013}. Even though stability is not proven for the proposed method, it follows intuitively from using the greedy approach in combination with the residual. If samples do not cluster at the boundaries, then the approximations used to calculate the residual become unstable. These unstable interpolants produce large errors near the boundary of the domain because of the Runge-phenomenon. This results in large residual values, which leads to sample placement near the boundary of the domain. Lastly, note that the sample locations from the NIPPAS are in general scattered, which complicates constructing an interpolant on these samples. This aspect will be discussed in more detail in section \ref{sec:ImplementationDetails}

\subsection*{Connection with empirical interpolation}~\\
\noindent Our NIPPAS procedure has similarities to the classical empirical interpolation procedure \cite{hesthaven_empirical_2016}. The main difference between general empirical interpolation and NIPPAS is the way in which the residual is constructed. In empirical interpolation the residual is based on the entire spatial/temporal surrogate $\tilde{\mbf{v}}$, rather than focussing on the QoI $\tilde{\mbf{u}}$. Consequently, the residual in empirical interpolation is defined as
\begin{equation}
\mbf{R}_{EI}(\mbf{z}) = \sum_{l=1}^{n_l} \mbf{g}_{l}(\mbf{z}, X)\circ L_l(\widetilde{\mbf{v}}(\mbf{z}))\ ,
\label{eq:REI}
\end{equation}
which has the advantage that one only needs to approximate a single term, i.e., $\tilde{\mbf{v}}(\mbf{z})$, instead of constructing several approximations $\widetilde{L_l(\mbf{v}(\mbf{z}))}$. However, the disadvantage is that the refinement measure $\mbf{R}_{EI}$ adds a degree of intrusiveness as it assumes that the operators $L_l$ are known completely \cite{hesthaven_efficient_2014}, which is not always the case and restricts the applicability. 

An advantage of the NIPPAS is that it focuses the sample placement on regions where $\tilde{\mbf{u}}$ needs to be refined, rather than where $\tilde{\mbf{v}}$ needs to be refined. The NIPPAS therefore focuses on areas in the parameter space that are most relevant for constructing an accurate surrogate for $\mbf{u}$, which results in faster convergence. The disadvantage is that the NIPPAS requires a small alteration of the black-box, which will be discussed in section 4.

\section{NIPPAS surrogate construction} ~\\
\label{sec:Method2}
\noindent The NIPPAS procedure when using refinement measure $\mbf{R}$ is shown in this section; the procedure when using $\mbf{R}_{\rho}$ is analogous. A schematic overview of the NIPPAS procedure is shown in figure \ref{fig:SchematicMethod}.

\subsection*{I. Placement of one random initial sample}~\\
\label{subsec:initialsamples}
\noindent The adaptive algorithm starts by performing initial sampling in the random space. Multiple initial sample configurations can be considered. It is advised to start the method with a single Monte-Carlo sample in the random space. This initialisation is used in the remainder of this paper.
%\begin{figure}[!h]
%\centering
%\includegraphics[width = \columnwidth]{FiguresMinimisingResidual_InitialSamples.pdf}
%\caption{\label{fig:InitialSamples} Initial sample location for 1D and 2D random space.}
%\end{figure}
The initial sample location and corresponding model evaluation is denoted as $(Z_0, U_0)=(z_0, \mbf{u}(z_0))$. More generic, a subscript $i$ indicates the number of adaptively placed samples, i.e., $Z_i=\lbrace \mbf{z}_j\rbrace_{j=0}^{j=i}$ and $U_i = \lbrace \mbf{u}(\mbf{z}_j)\rbrace_{j=0}^{j=i}$.

\subsection*{II-a. Non-intrusive computation of residual}~\\
\label{subsec:surrogaterefinement}
\noindent The refinement loop starts with computing the residual \eqref{eq:Residual}.

The residual requires the approximations $\widetilde{L_l(\mbf{v}(\mbf{z}))}$. These approximations can be easily constructed for a given sample set $Z_i$ by applying an interpolation operator $P$ (to be discussed in section \ref{subsec:initialsampleinterpolation}) on the values $\lbrace L_l(\mbf{v}(\mbf{z}_i))\rbrace_{j=0}^i$. However, \textit{the black-box solver in general only returns the solution values $\mbf{v}(\mbf{z}_i)$, and should be altered such that it also returns the values of the individual discretised differential operators $L_l(\mbf{v}(\mbf{z}_i))$. These discretised partial derivatives are used to compute the solution values within the black-box and may be output alongside the solution when sampling the black-box.} This requires only a slight alteration of the black-box, without changing the underlying PDE, and therefore keeps the methodology non-intrusive. To summarise, the approximations of the differential operator terms are given by
\begin{equation}
L_l(\mbf{v}(\mbf{z})) \approx \widetilde{L_l(\mbf{v}(\mbf{z}))} = P[(Z_i, (L_lV)_i]\ ,\ \ l=1,...,n_l\ ,
\label{eq:DifferentialOperatorApproximation}
\end{equation}
where
\begin{equation}
(L_lV)_i = \left\lbrace L_l(\mbf{v}(\mbf{z}_0)),..., L_l(\mbf{v}(\mbf{z}_i))\right\rbrace\ ,\ \ l=1,...,n_l\ ,
\end{equation}
which are the values that should be returned by the black-box solver alongside the solution values $\mbf{v}$. After approximating each differential operator term in the random space through interpolation, the interpolants are substituted into \eqref{eq:Residual}, which returns a function in the variable $\mbf{z}$. Notice that if the black-box solves \eqref{eqPD:ModelProblem} with negligible round-off and iteration error, then we satisfy
\begin{equation}
\mbf{R}(\mbf{z}_i) = 0,\ i=0,...,i\ ,
\end{equation}
which attains local maxima between the samples, as is shown in figure \ref{fig:SchematicMethod}. 

\subsection*{II-b. Stopping criterion based on the residual}~\\
\label{subsec:stoppingcriterion}
\noindent The refinement loop has to be terminated after a number of iterations.

For this purpose we need a stopping criterion, which reflects the quality of the surrogate model. Equation \eqref{eqcor:errorbound} can be used as a stopping criterion. However, \eqref{eqcor:errorbound} does not hold for non-linear PDEs and computing the required norms would be intractable in general. Ideally, the residual will show a similar convergence as the error in the surrogate. If this is the case, the magnitude of the residual is not only useful to adaptively sample the black-box, but it also serves as a reliable indication for the quality of the surrogate. Therefore, the refinement loop is stopped when the following criterion is met:
\begin{equation}
\|\mbf{R}(\mbf{z})\|_2 < \varepsilon\ ,
\label{eq:stoppingcriterion}
\end{equation}
where $\varepsilon$ is a specified threshold.

\subsection*{II-c. Find the location where the residual attains maximum}~\\
\noindent If criterion \eqref{eq:stoppingcriterion} is not met, the surrogate is refined by placing an extra sample according to \eqref{eq:Residual}. In other words, we need to solve the following global optimisation problem in $I_{\mbf{z}}$:
\begin{equation}
\mbf{z}_{\text{max}} = \arg \max_{\mbf{z}\in I_z} \|\mbf{R}(\mbf{z})\|_2\ .
\label{eq:GOP}
\end{equation}
The complexity of this global optimisation problem depends on characteristics of the objective function $\mbf{R}(\mbf{z})$. The residual is in general not smooth and has multiple local maxima. Therefore, in order to solve \eqref{eq:GOP}, a particle swarm method from the Global optimisation toolbox in Matlab is used, which is able to solve non-smooth global optimisation problems.

The solution of the global optimisation problem \eqref{eq:GOP} is denoted $\mbf{z}_{\text{max}}$ and is used in the next step to refine the surrogate.

\subsection*{II-d. Sample the model at the new sample location}~\\
\noindent To refine the surrogate model, we add the new sample $\mbf{z}_{\text{max}}$ to our current sample set $Z_i$: 
\begin{align}
Z_{i+1} &= Z_i\cup \mbf{z}_{\text{max}}\ ,\\
U_{i+1} &= U_i\cup \mbf{u}(\mbf{z}_{\text{max}})\ ,\\
(L_lV)_{i+1} &= (L_lV)_i\cup L_l(\mbf{v}(\mbf{z}_{\text{max}}))\ ,\ \ l=1,...,n_l\ .
\end{align}
The new sample set $(L_lV)_{i+1}$ is suited for constructing an improved approximation $\widetilde{L_l(\mbf{v})}$, see equation \eqref{eq:DifferentialOperatorApproximation}, which is used in the next iteration of the refinement loop.

\subsection*{III Final surrogate is constructed by interpolation}~\\
\noindent After the stopping criterion \eqref{eq:stoppingcriterion} is met, the refinement loop terminates and we end up with a sample set $Z_n$ and an evaluation set $U_n$. In order to construct the final surrogate, we apply the interpolation operator $P$ on the final sample set, leading to
\begin{equation}
\widetilde{\mbf{u}}(\mbf{z}) = P[(Z_n, U_n)]\ ,
\end{equation}
where $\widetilde{\mbf{u}}$ is expected to be an accurate approximation to $\mbf{u}$. If wanted, statistical quantities like \eqref{eq:meandef} can be computed by using this surrogate to compute the desired integrals.

\section{Implementation details}~\\
\label{sec:ImplementationDetails}
\noindent In this section some implementation details are discussed.

\subsection*{Surrogate modelling by polynomial interpolation for scalar QoIs}~\\
\label{subsec:initialsampleinterpolation}
\noindent Interpolation for a scalar QoI is discussed here. Interpolation for a vector QoI is achieved by applying the interpolation operator to each element of the QoI vector individually.

Assume we have a sample set $Z = \{\mbf{z}_i\}_{i=0}^N$ and corresponding QoI values $U = \{u(\mbf{z}_i)\}_{i=0}^N$. As mentioned in section \ref{sec:Method2}, we denote by $P$ the interpolation operator which acts on the set $(Z, U)$. Polynomial interpolation aims to find a polynomial $\tilde{u}$, such that:
\begin{equation}
\tilde{u}(\mbf{z}_i) = P[(Z, U)](\mbf{z}_i) = u(\mbf{z}_i)\ .
\label{eq:Interpolation}
\end{equation}
The polynomial $\tilde{u}$ serves as a surrogate for $u(\mbf{z})$. Finding $\tilde{u}$ is straightforward in a 1D random space, but interpolation in multi-dimensional random spaces is not unique, and the result is influenced by the choice of interpolation basis. Additionally, the sample locations from our adaptive sampling are scattered, which further complicates the interpolation procedure. In order to make the interpolation basis unique, graded lexicographic ordered Chebyshev polynomials $\{\phi_i(\mbf{z})\}_{i=0}^N$ are used, which are known for their stability when using them for interpolation \cite{trefethen_approximation_2013}. The interpolant is found by solving a Vandermonde system \cite{xiu_numerical_2010}:
\begin{equation}
\underbrace{\left( \begin{array}{ c c c}
\phi_0(\mbf{z_0}) & \hdots & \phi_N(\mbf{z_0}) \\ 
\vdots &  & \vdots \\ 
\phi_0(\mbf{z_N}) & \hdots & \phi_N(\mbf{z_N})
\end{array}\right)}_{A} \vecbrace{c_0 \\ \vdots \\ c_{N}} = \vecbrace{u(\mbf{z}_0) \\ \vdots \\ u(\mbf{z}_N)}\ ,
\label{eq:vandermondesystem}
\end{equation}
which results in the interpolant
\begin{equation}
\tilde{u}(\mbf{z}) = \sum_{i=0}^N c_i \phi_i(\mbf{z})\ .
\end{equation}
The Vandermonde matrix $A$ can become ill-conditioned when increasing the number of samples \cite{higham_accuracy_2002}, which makes solving \eqref{eq:vandermondesystem} difficult. Nevertheless, other approaches that circumvent solving \eqref{eq:vandermondesystem} \cite{narayan_stochastic_2012, boor_multivariate_1990, sauer_polynomial_1997} do not have the flexibility to reuse the inverse computations of previous iterations for solving the larger system at the current iteration. This leads to a significant increase in computational expense as the adaptive sampling placement is done iteratively. Therefore, Vandermonde interpolation is used in the NIPPAS method, and a robust procedure for solving the possibly ill-conditioned system \eqref{eq:vandermondesystem} should be used. 

Several methods exist for solving such ill-conditioned systems of equations: regularisation \cite{neumaier_solving_1998}, singular value decomposition \cite{varah_numerical_1973}, pivoted $LU$-factorisation \cite{higham_accuracy_1989}, and pseudo-inversion \cite{klinger_approximate_1968}. In practice, computing the pseudo-inverse is often not advised due to its computational cost, which is $O(N^3)$. However, our adaptive sampling algorithm requires the solution of multiple linear systems with the same Vandermonde matrix $A$ at each iteration, which makes the use of a pseudo-inverse beneficial. Nevertheless, computing the full pseudo-inverse at each iteration would be inefficient and therefore Greville's algorithm \cite{mohideen_recursive_1991} is used to update the pseudo-inverse after adding a new row/column to $A$. Consequently, interpolating polynomials can be constructed without solving the full linear system \eqref{eq:vandermondesystem} and without computing the full pseudo-inverse each time the interpolation operator $P$ is applied. The implementation of Greville's algorithm is discussed in more detail in the next subsection.

\subsection*{Rank-one update of pseudo-inverse}~\\
\noindent As mentioned before, when adding a new sample to the sample set, an extra row and column are appended to the existing Vandermonde matrix \eqref{eq:vandermondesystem}. Therefore, a new pseudo-inverse needs to be determined for this new Vandermonde matrix \cite{mohideen_recursive_1991}. Instead of calculating the new entire inverse, Greville's algorithm is used to iteratively compute the pseudo-inverse of a given matrix $A$.

Assume we have a matrix $A\in\mathbb{R}^{N\times N}$ and its corresponding pseudo-inverse $G$. When a column is added to $A$, i.e.,
\begin{equation}
A' = [A\ a]\ ,
\end{equation}
then Greville's algorithm computes the pseudo-inverse of the new matrix $A'$ recursively. In more detail, the new pseudo-inverse $G'$ of $A'$ is given by \cite{shinozaki_numerical_1972}
\begin{equation}
G' = \left(\begin{array}{c}
G - db^T\\
b^T
\end{array}\right)\ ,
\end{equation}
where
\begin{align}
a^{(1)} &= AGa\ ,\ \ a^{(2)} = a - a^{(1)}\ ,\\
d &= Ga\ ,\\
b^T &= \left\lbrace\begin{array}{l l}
\frac{\left(a^{(2)}\right)^H}{\left(a^{(2)}\right)^Ha^{(2)}}\ ,& \text{ if } a^{(2)} \neq \mbf{0}\ ,\\
(1+d^Td)^{-1}d^TG\ ,& \text{ if } a^{(2)} = \mbf{0}\ ,
\end{array}\right.
\end{align}
where the $H$ denotes the conjugate transpose. Notice that our refinement loop adds both a row and a column to the Vandermonde matrix \eqref{eq:vandermondesystem}, each time a new sample is added. Therefore the procedure described above has to be performed twice, first adding a column $A' = [A\ c]$, then adding a column $r^T$ to the transpose of the resulting matrix, i.e., $(A'')^T = ((A')^T\ r^T)$.

A recursive algorithm for calculating the pseudo-inverse is necessary to reduce the algorithmic complexity of the adaptive sampling procedure. For an $N\times N$ matrix, Greville's algorithm performs an inverse update with complexity $O(8N^2)$, while computing the full inverse directly has complexity $O(N^3)$ \cite{mohideen_recursive_1991} and becomes infeasible quickly in high dimensional spaces $I_{\mbf{z}}$, where the number of samples required for constructing an accurate surrogate is high.

\subsection*{Algorithmic complexity scales well for high dimensions}~\\
\noindent The algorithmic complexity gives an estimate on how the computational effort scales with the number of samples and the dimension of the random space. 

The computational complexity is determined by the complexity of the individual parts; updating pseudo-inverse, interpolation, and particle swarm optimisation. First we discuss the algorithmic complexity of a single iteration in the refinement loop.

Updating the pseudo-inverse is achieved by using Greville's algorithm. Assume a row and column are added to an existing pseudo-inverse of dimensions $i\times i$. Using Greville's algorithm, this can be done in $O(8i^2)$ operations \cite{mohideen_recursive_1991} (discussed in section \ref{sec:ImplementationDetails}). 

The interpolation procedure is fairly cheap when the pseudo-inverse is available. In detail, when interpolating on $i$ sample locations, solving \eqref{eq:vandermondesystem} only requires a matrix-vector multiplication of $O(i^2)$ operations.

The complexity of the particle swarm optimisation is determined by the number of particles used and the maximum number of iterations. In our case we use a particle swarm of $N_p$ particles and a maximum of $N_{it}$ iterations. Typical values for $N_p$ and $N_{it}$ are $100\dim(I_{\mbf{z}})$ and $200\dim(I_{\mbf{z}})$, respectively. In the worst case, the number of maximum iterations is reached and we have to evaluate the residual at $N_pN_{it}$ locations, without the guarantee to have found an optimum. Moreover, if we have $i$ samples, then the residual is a polynomial of degree $i-1$ and can be evaluated with Horner's \cite{higham_accuracy_2002} method in $O(i)$ operations. This results in a total cost of the particle swarm optimisation for a single iteration of the surrogate construction of $O(i\dim(I_{\mbf{z}})^2 n_l)$, where $n_l$ is the number of terms in the PDE \eqref{eqPD:ModelProblem}.

The total cost of each iteration in the refinement loop is now given by:
\begin{equation}
O(\underbrace{i^2}_{\text{Greville}} + \underbrace{i^2}_{\text{interpolation}} + \underbrace{i\dim(I_{\mbf{z}})^2 n_l}_{\text{particle swarm optimisation}})\ .
\end{equation}
Hence, if we run the refinement loop $N$ times, a conservative upper bound for the total algorithm complexity becomes
\begin{equation}
O(N^3 + N^2\dim(I_{\mbf{z}})^2 n_l)\ .
\end{equation}
Notice that the complexity depends quadratically on the dimension of the random space. However, the number of samples necessary for achieving a specified accuracy in higher-dimensional random space will generally increase with the number of dimensions, and computational cost will therefore increase faster than quadratic with the dimension of the random space.

\section{Results}~\\
\label{sec:Results}
\noindent In this section we present multiple examples that illustrate the efficiency of our method.  In order to study and compare convergence properties, two error measures are defined, one for the error in the statistical quantities, i.e., mean, variance, etc., and one for the error in the surrogate. 

The error in the $k$-th statistical moment is the 2-norm of the difference vector between the exact and the approximate statistical moment as computed in \eqref{eq:meandef}:
\begin{equation}
e_{\rho}^{(k)} := \|\mathbb{E}[\mbf{u}^k - \tilde{\mbf{u}}^k]\|_2 = \left\|\int_{I_\mbf{z}}\rho(\mbf{z})(\mbf{u}^k(\mbf{z}) -   \tilde{\mbf{u}}^k(\mbf{z}))d\mbf{z}\right\|_2\ ,
\label{eq:ErrorStatisticalMoments}
\end{equation}
where the integral is calculated using numerical integration with negligible error, i.e., a tensor based Gauss quadrature rule with 100 nodes in each direction. The uncertainties are assumed to be uniformly distributed, $\rho(\mbf{z})=1$, unless stated otherwise. 

Secondly, the error in the surrogate is computed as follows:
\begin{equation}
e := \frac{1}{N_{MC}}\sum_{i=1}^{N_{MC}} \| \mbf{u}(\mbf{z}_i) - \tilde{\mbf{u}}(\mbf{z}_i)\|_2\ ,
\label{eq:L2Error}
\end{equation}
where the sum is taken over a large number of uniform Monte-Carlo samples ($N_{MC}$) in the random space.

The first two test-cases consider a steady and unsteady advection-diffusion equation, respectively. They demonstrate the difference between our approach and conventional empirical interpolation, effect of discretisation on error convergence, difference between refinement measures \eqref{eq:Residual} and \eqref{eq:ResidualPDF}, and applicability to non-hypercube domains. The third and last test case considers the non-linear shallow water equations and demonstrates how our method can be used in combination with a new state-of-the-art neural network based PDE solver.

\subsection*{Steady-state advection-diffusion equation}~\\
\noindent In this part we consider a test-case such that the assumptions in theorem 1 hold and we study the following:
\begin{itemize}
\item Comparison of refinement measure \eqref{eq:Residual} and \eqref{eq:RExact}.
\item Asymptotic sample distribution.
\item Comparison between NIPPAS and conventional empirical interpolation.
\end{itemize}
The underlying PDE is the dimensionless steady state advection-diffusion problem given by:
\begin{equation}
Re(z)v_x - v_{xx} = 0\ ,\ v(0, z) = 0,\ v(1, z)=1,\ x\in[0, 1]\ ,
\end{equation}
where $Re(z)$ is the Reynolds number, which is assumed to be uncertain and a function of $z$. The equations are discretised using a finite-difference approach on an equidistant grid with a resolution of $\Delta x$ with $N_{PDE}$ grid points. The solution vector on the computational grid $\mbf{v}(z) \approx v(\mbf{x})$, with $x_i = i\Delta x$ for $i=1,...,N_{PDE}$, is obtained by solving the following linear system:
\begin{equation}
Re(z)L_1\mbf{v} - L_2\mbf{v} = \mbf{S}(z)\ ,
\end{equation}
where
\begin{align}
L_1 &= \frac{1}{2\Delta x}\left(\begin{array}{ccccc}
0&1& & & \\
-1&0&1 & & \\
 &\ddots & \ddots & \ddots & \\
 & & -1& 0 &1\\
  & & 0& -1& 0
\end{array}\right),\ 
L_2 = \frac{1}{\Delta x^2}\left(\begin{array}{ccccc}
-2& 1& & & \\
1&-2&1 & & \\
 &\ddots & \ddots & \ddots & \\
 & & 1& -2 &1\\
  & & 0& 1& -2
\end{array}\right)\\
\mbf{S}(z) &= \left(\begin{array}{c}
0\\
\vdots\\
\vdots\\
0\\
-\frac{Re(z)}{2\Delta x} - \frac{1}{\Delta x^2}
\end{array}\right)\ ,
\end{align}
where the boundary conditions enter the discretised equation via the vector $\mbf{S}$. Notice that the discretised PDE satisfies the assumptions of theorem 1. Example solutions for different Reynolds numbers are shown in figure \ref{fig:SSAD_Example}.
\begin{figure}[!h]
\centering
\includegraphics[width =0.5\textwidth]{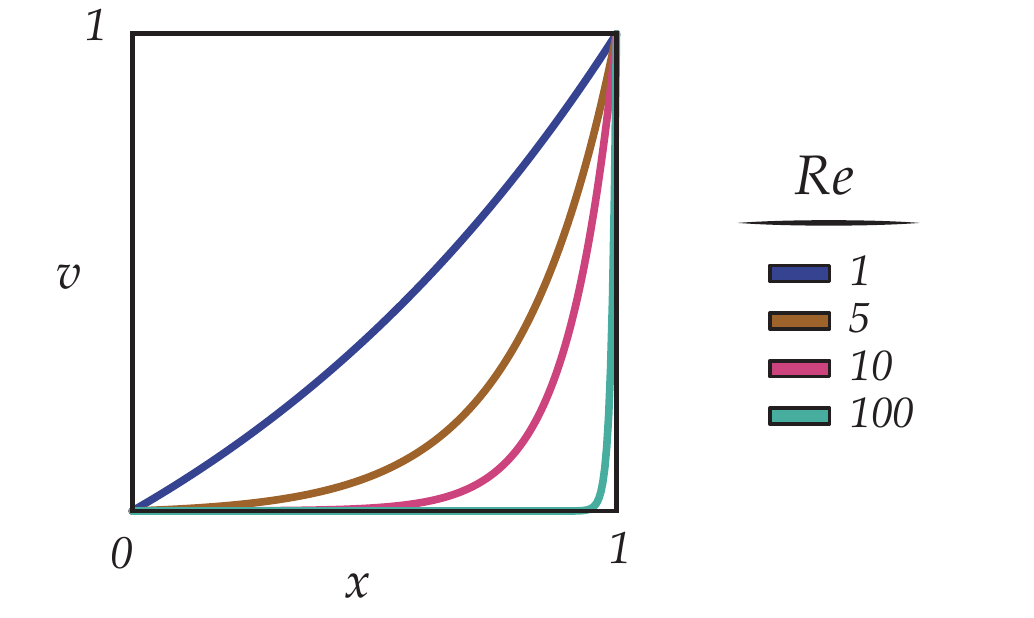}
\caption{\label{fig:SSAD_Example} Example solutions for different Reynolds number. The solution are computed on a computational grid with $\Delta x = 10^{-3}$.}
\end{figure}\\
In order for the discretisation to produce stable results, the cell Reynolds number $Re_{\Delta x} = Re\Delta x$ should satisfy $Re_{\Delta x}<2$, which is satisfied by picking $\Delta x$ sufficiently small, which is $\Delta x = 10^{-3}$ in our case ($N_{PDE} - 1000$).\\

\noindent\textit{{There is no significant difference in convergence between refinement measures $\mbf{R}$ and $\mbf{R^*}$}}~\\\\
The difference between refinement measures $\mbf{R}$ and $\mbf{R}^*$ is the incorporation of the scaling term $\left(\sum_{l=1}^{n_l} G_l(z, X)L_l\right)^{-1}$ in $\mbf{R}^*$. This scaling term is in general expensive to compute as a function of $z$ and is therefore often infeasible to incorporate in the refinement measure. However, for this simple test-case we can compute this scaling term as a function of $z$ and are able to study its effect on sample placement. In order to compare the two refinement measures $\mbf{R}$ \eqref{eq:Residual} and $\mbf{R}^*$ \eqref{eq:RExact}, three different functional forms for $Re(z)$ are used, where $z$ is uniformly distributed between $[0, 1]$. The three different functional forms are given by:
\begin{align}
Re_1(z) &= 99z+1\ ,\\
Re_2(z) &= 10^{2z}\ ,\\
Re_2(z) &= 10^{-2(z-1)}\ ,
\end{align}
and are shown in figure \ref{fig:SSAD_Results1}. All three functions range from $1$ to $100$ for $z\in[0, 1]$, but have completely different shapes, which influence the scaling function and therefore the sample locations when using refinement measure $\mbf{R}^*$.

A comparison between the two refinement measures $\mbf{R}$ and $\mbf{R}^*$ is made by choosing $Q=I$, i.e., we are interested in the entire solution $\mbf{u} = \mbf{v}$ as a function of $z$. The adaptive sample placement starts by placing a single uniformly distributed sample in the random space. The solution values and derivative values $\mbf{v}_x:=L_1\mbf{v}$ and $\mbf{v}_{xx}:=L_2\mbf{v}$ are sampled at the sample location. These values are used to construct the interpolants which are required for computing $\mbf{R}_v$. The convergence of the error in the surrogate $\tilde{\mbf{v}}(z)$ for both refinement measures and the different $Re_i$ is shown in figure \ref{fig:SSAD_Results1}.
\begin{figure*}[!h]
\centering
\includegraphics[width = \textwidth]{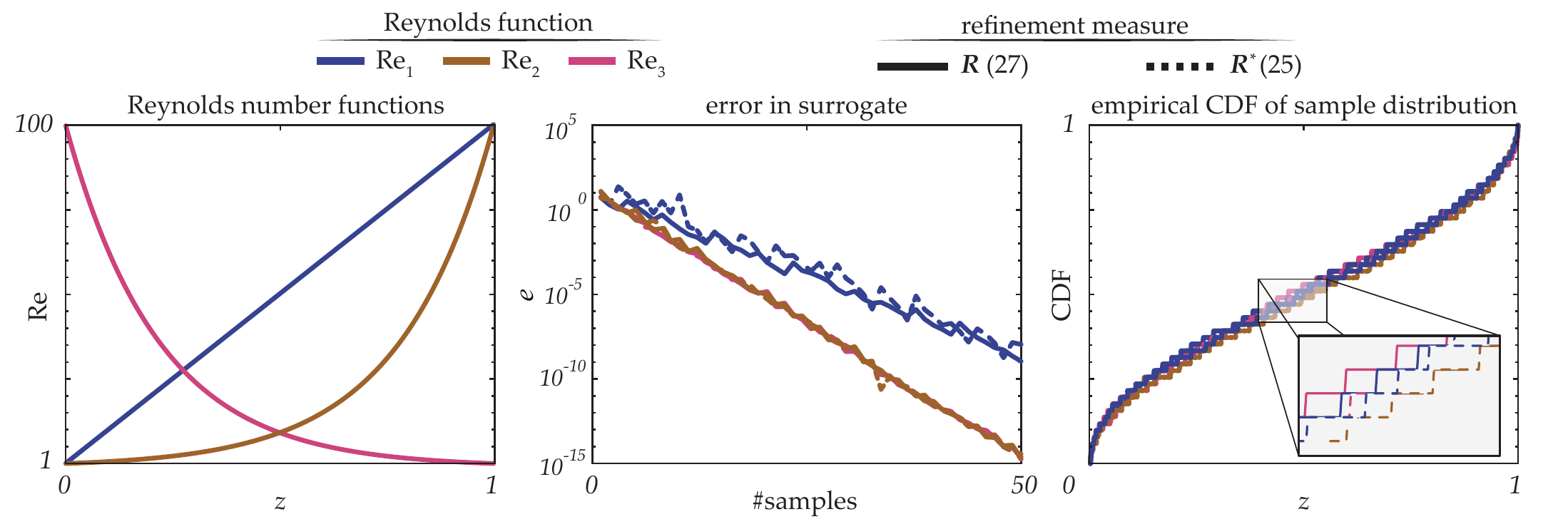}
\caption{\label{fig:SSAD_Results1} The results for the steady-state advection diffusion equation for the two refinement measures $\mbf{R}$ and $\mbf{R}^*$ and three different Reynolds functions in the random space. The solutions are computed using $\Delta x = 10^{-3}$.}
\end{figure*}\\
The error in the surrogate converges exponentially fast to machine precision for both refinement measures without any significant difference. Samples are placed at similar locations for both refinement measures, which is shown in the empirical CDF plot in figure \ref{fig:SSAD_Results1}. Regarding stability of the interpolant, a clustering of samples occurs at the boundaries of the random space, which stimulates stable interpolation. 

Refinement measure $\mbf{R}^*$ uses all the terms that depend on $z$ and should be optimal for sample placement, see equations \eqref{eqthm1:R} and \eqref{eq:RExact}. However, the error in the surrogate is not only determined by sample placement, but also by the polynomial interpolation procedure. Consequently the best refinement measure does not necessarily lead to the best convergence in the error. Including the scaling term in $\mbf{R}^*$ appears to have little effect on the sample placement, at least for this simple test-case. In fact, for this test-case one could refine directly on $\|\tilde{\mbf{u}} - \mbf{u}\|_2$ (because $\mbf{u}$ is explicitly known), and this gives the same results as refining based on $\mbf{R}^*$. As there is no significant difference between the two refinement measures, we use refinement measure $\mbf{R}$ in the remainder of this paper, as it is more generally applicable.\\

\noindent\textit{{NIPPAS converges faster than conventional empirical interpolation depending on the QoI}}~\\\\
We compare convergence of the surrogate for the NIPPAS based on \eqref{eq:Residual} with the conventional empirical interpolation based on \eqref{eq:REI}. The main difference (in absence of incorporation of the PDF) is the fact that empirical interpolation places new adaptive samples based on a residual which is based on the entire solution $\mbf{v}$, whereas the NIPPAS uses a residual based only on the QoI $\mbf{u}$. To compare both methods, the QoI is set to $u = (\mbf{v})_{500}$ (the solution computed at the middle of the computational grid, with $Q = (0, ..., 0, 1, 0, ..., 0)$), and the Reynolds number is given by $Re_1(z)$. The convergence comparison is shown in figure \ref{fig:SSAD_Results2}.
\begin{figure}[!h]
\centering
\includegraphics[width = 0.5\textwidth]{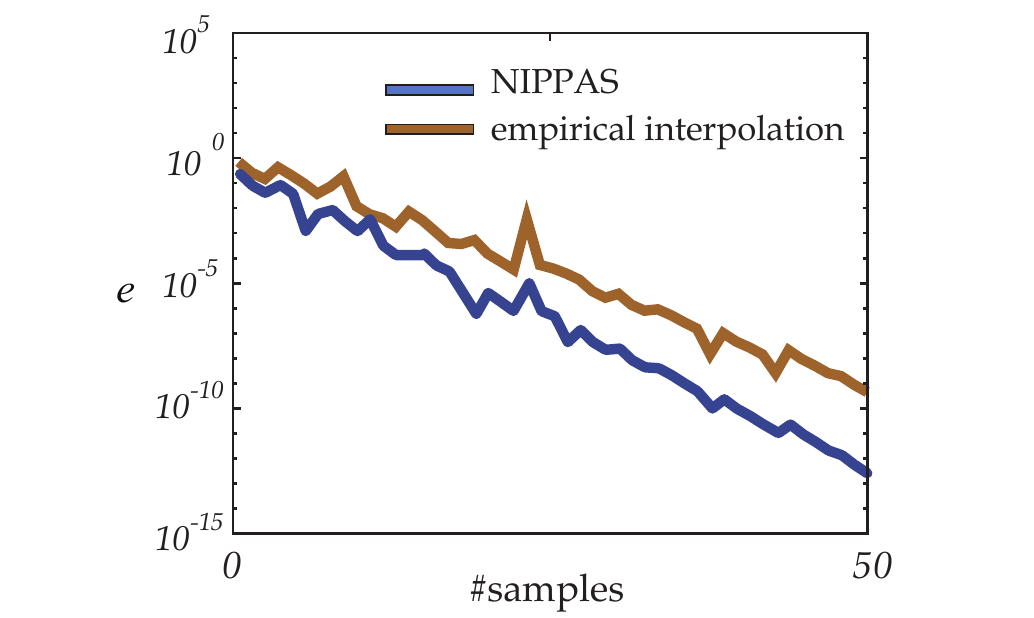}
\caption{\label{fig:SSAD_Results2} Convergence comparison for NIPPAS method and conventional empirical interpolation for the steady-state advection diffusion equation.}
\end{figure}\\
The NIPPAS enhances the surrogate by focusing on locations that are relevant for the QoI. This leads to faster convergence for the NIPPAS when compared to conventional empirical interpolation. However, in case more solution values from the solution vector $\mbf{v}$ would be incorporated, i.e., more non-zero entries in the columns of $Q$, the convergence speed-up will decrease and eventually the convergence coincides with the one from empirical interpolation when the QoI uses the entire solution vector $\mbf{v}$.

\subsection*{Unsteady advection-diffusion equation}~\\
\noindent In the previous example we applied NIPPAS to a steady-state PDE for a 1D random space in order to compare refinement measures and to compare with conventional empirical interpolation. In this section we study the following:
\begin{itemize}
\item The effect of the discretisation method applied to the PDE.
\item The accuracy of NIPPAS for approximating statistical moments in a 2D random space.
\item The construction of a surrogate model on non-hypercube domains.
\end{itemize}
Therefore, we thoroughly study the NIPPAS for an unsteady advection-diffusion equation with two parameter uncertainties. We start with a 2D hypercube random space and then gradually increase the complexity of the test-case. 

The underlying PDE is the 1D advection-diffusion equation, given by
\begin{equation}
v_t + z_1 v_x = z_2 v_{xx}\ ,
\label{eq:Advectiondiffusion}
\end{equation}
where the advection parameter $z_1$ and the diffusion parameter $z_2$ are assumed to be uncertain and uniformly distributed between $[0, 2\pi]$. For the problem \eqref{eq:Advectiondiffusion} to be well-posed, initial and boundary conditions are required. A spectral spatial discretisation method \cite{canuto_spectral_2012} is used for solving \eqref{eq:Advectiondiffusion} and the boundary conditions are taken periodic for $x\in[0, 2\pi]$ and the initial condition is given by
\begin{equation}
v(x, 0) = v_0(x) = \sin(x)\ .
\label{eq:advectiondiffusionintialcondition}
\end{equation}
The spectral spatial discretisation is performed on an equidistant grid $x_i = (2\pi i)/N_x$ for $i=0,...,N_x$, where $N_x = 256$. This results in a solution vector $\mbf{v}(t)$:
\begin{equation}
\mbf{v}(t) = \vecbrace{v_0(t) \\ \vdots \\ v_{N_x}(t)}\ ,
\end{equation}
where $v_i(t)$ is the approximate solution at grid point $x_i$. Additionally, taking derivatives can be written in terms of a matrix-vector multiplication
\begin{align}
v_x &\approx D_x \mbf{v}\ ,\\ 
v_{xx} &\approx (D_x)^2 \mbf{v}\ ,
\end{align}
where $D_x$ is the spectral differentiation matrix \cite{canuto_spectral_2012}. As a result, the solution vector $\mbf{v}(t)$ can be obtained by solving the semi-discrete problem
\begin{equation}
\mbf{v}_t + z_1 D_x \mbf{v} = z_2 D^2_x \mbf{v}\ ,
\label{eq:SDproblem}
\end{equation}
which can be rewritten in the form \eqref{eqPD:ModelProblem} by discretising the time-derivative and formulating the resulting set of equations in matrix form. Notice that this results in a block-diagonal system.

A surrogate is created for the quantity $u(z_1, z_2) = v_{N_x}(1)$, which is the solution at the right boundary of the physical domain at $t=1$.\\

\noindent\textit{{Effect of discretisation method is small}}~\\\\
As we use a spectral method with a fine resolution for the spatial discretisation and the problem is linear, the time discretisation error is expected to be dominant, and therefore the effect of the time discretisation method is studied. The semi-discrete problem \eqref{eq:SDproblem} is solved using different time discretisation schemes, i.e., backward-Euler, Crank-Nicolson and fourth-order explicit Runge-Kutta (RK4), where we fix the time step at $\Delta t = 10^{-5}$. 

To demonstrate the efficiency of the NIPPAS method, convergence is compared to a nested Smolyak grid \cite{klimke_algorithm_2005} on the Clenshaw-Curtis nodes. The errors in the surrogate for both the NIPPAS method and the Smolyak solution are computed by using a Monte-Carlo reference solution based on $5000$ samples. Notice that the reference solution changes, depending on the time discretisation. The reference solution for a Crank-Nicolson time discretisation is shown in figure \ref{fig:advdiffMCreference}. 
\begin{figure}[H]
\centering
\includegraphics[width = 0.49\textwidth]{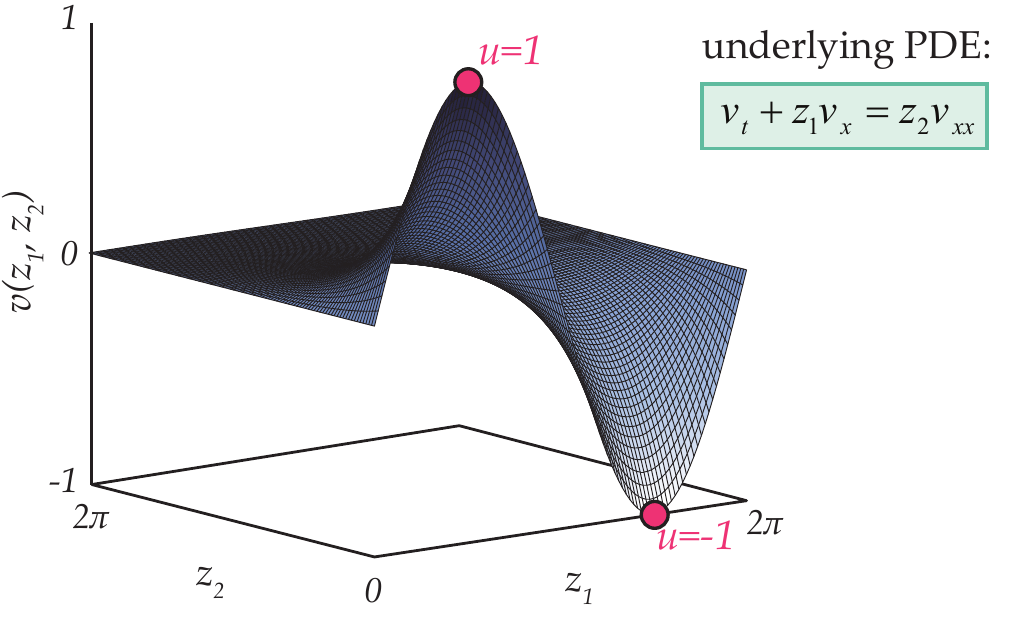}
\caption{\label{fig:advdiffMCreference} Reference solution for the advection-diffusion equation for the quantity $u(z_1, z_2) = v_{N_x}(1)$ based on 5000 samples.}
\end{figure}~\\
The choice of the time discretisation method affects the accuracy of the black-box solver, and changes the shape of the surrogate slightly, as different time discretisations do not give the same QoI at the exact same location in random space. As a result, when changing the discretisation, a new reference solution has to be computed to study convergence. To further clarify, there is a difference between the exact/wanted surrogate, which is obtained by solving \eqref{eqPD:ModelProblemexact}, and the discrete exact surrogate, obtained by solving the discretised equations \eqref{eqPD:ModelProblem}. This difference is shown in the following equation:
\begin{equation}
\underbrace{\|u-\tilde{u}_{N_x}\|_2}_{\text{error w.r.t. exact solution}}\leq \underbrace{\|u-u_{N_x}\|_2}_{\text{discretisation error}} + \underbrace{\|u_{N_x} - \tilde{u}_{N_x}\|_2}_{\text{error w.r.t. discrete solution}}\ ,
\label{eq:ADerror}
\end{equation}
where $u$ is the exact solution (solution of \eqref{eqPD:ModelProblemexact}), $u_{N_x}$ is the discrete solution computed using $N_x+1$ grid points (solution of \eqref{eqPD:ModelProblem}), and $\tilde{u}_{N_x}$ is the surrogate based on the discrete solutions. For this reason, we plot both the error in the surrogate with respect to the exact solution and the discrete solution. The results are shown in figure \ref{fig:AdvDiffTimeDiscrResults}.
\begin{figure}[!h]
\centering
\includegraphics[width = \textwidth]{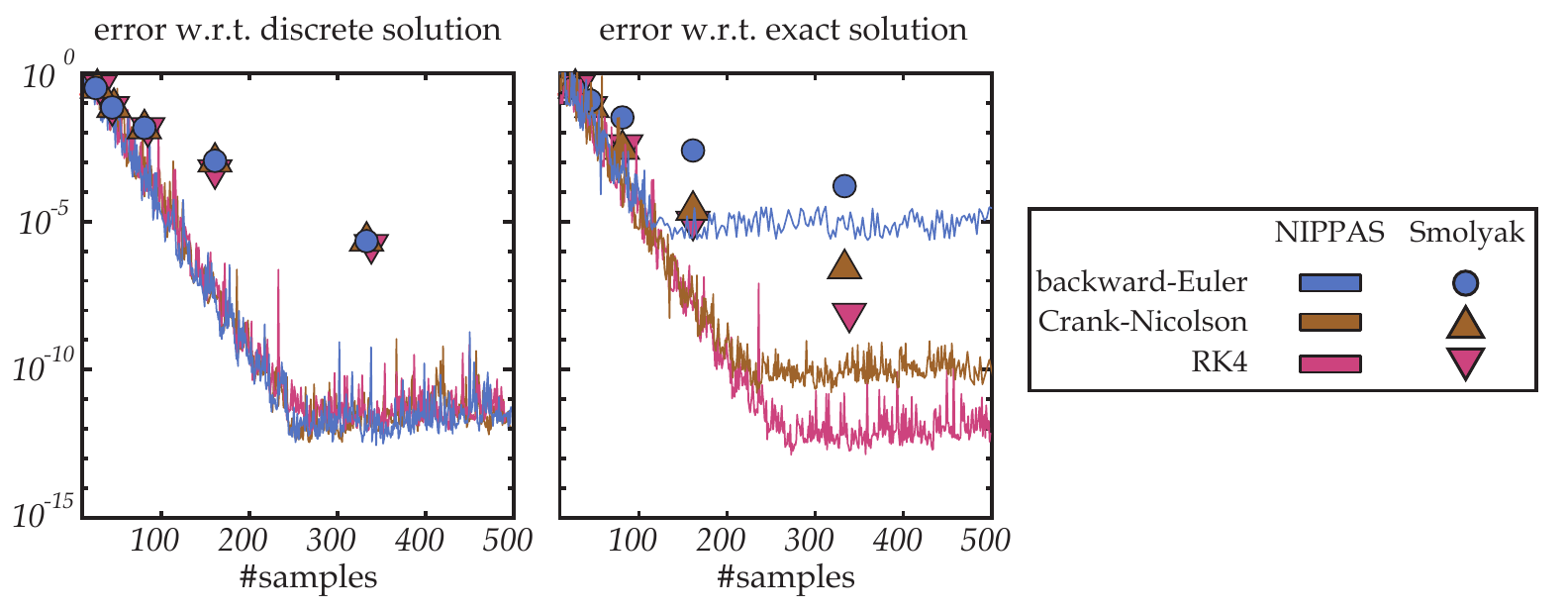}
\caption{\label{fig:AdvDiffTimeDiscrResults} Error \eqref{eq:L2Error} comparison between the Smolyak sparse grid surrogate and the NIPPAS surrogate for different time discretisation methods with $\Delta t = 10^{-5}$ and $N_x=256$. The errors are plotted with respect to both the discrete solution of \eqref{eq:Advectiondiffusion} and the exact solution of \eqref{eq:SDproblem}.}
\end{figure}~\\
The error with respect to the discrete solution converges to zero, with a convergence rate that is significantly faster than the Smolyak sparse grid. Convergence is non-monotonic, which is common for adaptive sampling methods \cite{narayan_adaptive_2014}. The error with respect to the exact solution stalls before machine precision is reached, which is due to the discretisation error. To clarify, the error with respect to the discrete solution converges to zero and equation \eqref{eq:ADerror} therefore states that the observed error after stalling is the discretisation error. 

The convergence behaviour of the error with respect to the discrete solution is not dependent on the time discretisation, which indicates that the performance of our method does not depend on the underlying discretisation.\\

\noindent\textit{{Faster convergence for statistical quantities with PDF weighing}}~\\\\
Next, to study the effect of the two different refinement measures \eqref{eq:Residual} and \eqref{eq:ResidualPDF}, we now assume independent $\beta$-distributed input uncertainties $z_1$ and $z_2$:
\begin{equation}
\rho(z_1, z_2) = c\cdot(2\pi z_1)^{\alpha_1}(2\pi z_2)^{\alpha_2}(1-2\pi z_1)^{\beta_1}(1-2\pi z_2)^{\beta_2}\ ,
\label{eq:betadistr}
\end{equation}
where $(\alpha_i,\beta_i)$ are parameters that characterise the PDF and the constant $c$ is chosen such that the PDF has total probability 1 on $[0, 2\pi]^2$. Convergence behaviour is studied as a function of the shape of the underlying PDF and therefore we vary the PDF parameters as follows
\begin{equation}
\alpha_i\in\{1, 2, 3, 4, 5\},\ \beta_i\in\{1,2,3,4,5\},\ i=1,2\ ,
\end{equation}
which results in a total of 625 PDFs with totally different shapes. The convergence statistics of the mean, variance and the entire surrogate are shown in figure \ref{fig:AdvDiffPDF}.
\begin{figure*}[!h]
\centering
\includegraphics[width = \textwidth]{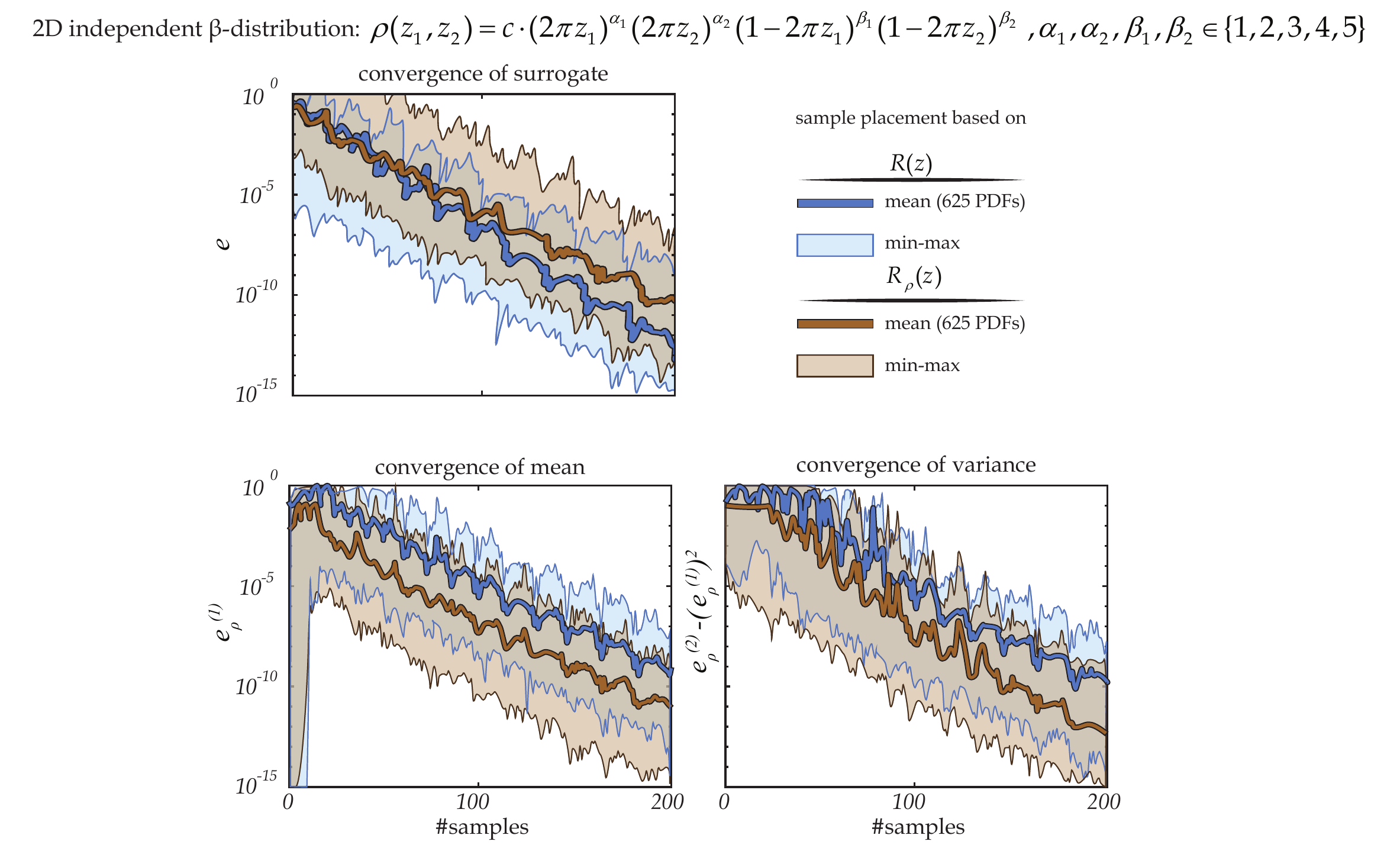}
\caption{\label{fig:AdvDiffPDF} Convergence in advection-diffusion surrogate, mean and variance for the two refinement measures \eqref{eq:Residual} and \eqref{eq:ResidualPDF}. The error in the surrogate is computed  with \eqref{eq:L2Error} with 5000 Monte-Carlo samples. The errors in the mean and variance are calculated with \eqref{eq:ErrorStatisticalMoments}.}
\end{figure*}\\
Figure \ref{fig:AdvDiffPDF} shows the average convergence behaviour taken over the 625 different PDFs. From the figure we may conclude that weighing the residual with the PDF results in slower convergence for the surrogate model, see equation \eqref{eq:L2Error}, but leads to faster convergence in both the mean and the variance, which are computed using \eqref{eq:ErrorStatisticalMoments}. Furthermore, the variation in convergence over all 625 different PDFs, given by the shaded area around the mean line, has similar width for both refinement measures \eqref{eq:Residual} and \eqref{eq:ResidualPDF}. The choice of refinement measures thus depends on the type of convergence the user requires, i.e., convergence in the surrogate or convergence in the statistical quantities. Notice that the shaded area for the convergence in the mean has a minimum that starts at 0 initially, which is due to the fact that for some PDFs the exact mean is zero, and as the initial sample is placed at a location in the random space where the response is also 0, we start with an approximate mean that is equal to the exact mean. However, the sample placement does not stop immediately, as the stopping criterion \eqref{eq:stoppingcriterion} is not met.\\

\noindent\textit{{Method shows fast convergence for non-hypercube domains}}~\\\\
The 2D $\beta$-distribution \eqref{eq:betadistr} just discussed, comprises two independent uncertainties and therefore the space in which we construct a surrogate is a hypercube. However, when dealing with dependent uncertainties, the space in which the surrogate is constructed, is in general not a hypercube, and many existing UQ methods fail when dealing with a complex random space.

To study the performance of the NIPPAS method on more general geometries, we assume a non-hypercube domain $D\subset\mathbb{R}^d$ with associated uniform PDF $\rho(\mbf{z})$. In order to sample the model on the domain $D$, we restrict the residual \eqref{eq:Residual} to $D$ by multiplying it with an indicator function:
\begin{equation}
R_D(\mbf{z}) = R(\mbf{z})\mathbb{I}_D(\mbf{z})\ , 
\end{equation}
where
\begin{equation}
\mathbb{I}_D(\mbf{z}) = \left\lbrace \begin{array}{l l}
1,\ & \text{if}\ z\in D\ ,\\
0,\ & \text{otherwise}\ .
\end{array}\right.
\end{equation}
After altering the residual definition slightly and placing an initial sample randomly in the non-hypercube random space, the NIPPAS method can be applied straightforwardly. Note that the basis functions are the same for the hypercube case, which are given by the Chebyshev polynomials defined on the smallest hypercube that comprises $D$. In order to show the efficiency of our method for non-hypercube domains, we again construct a surrogate for the response shown in figure \ref{fig:advdiffMCreference}, but now on more complexly shaped domains. Three different geometries with different characteristics are chosen:
\begin{itemize}
\item sharp corners:\\
\[(z_2\geq 0)\cap(z_2 - \sqrt{3}z_1 \leq 0)\cap(\sqrt{3}z_1 + z_2 \leq 2\pi\sqrt{3})\ ,\]
\item smooth boundary:\\
\[(z_1 - \pi)^2 + (z_2 - \pi)^2\leq \pi^2,\]
\item domain with holes:
\begin{align*}
&(z_1, z_2)\in[0,2\pi]^2,\ \text{but}\\
&\neg\left((z_1 - \pi)^2 + (z_2 - \pi)^2\leq \frac{\pi^2}{4}\right)\cup\\ 
&\neg\left((z_1 - \frac{\pi}{3})^2 + (z_2 - \frac{\pi}{3})^2\leq \frac{\pi^2}{25}\right)\cup\\
&\neg\left((z_1 - \frac{5\pi}{3})^2 + (z_2 - \frac{5\pi}{3})^2\leq \frac{\pi^2}{9}\right)\ .
\end{align*}
\end{itemize}
The convergence results are shown in figure \ref{fig:AdvDiffComplexDomain}.
\begin{figure*}[!h]
\centering
\includegraphics[width = \textwidth]{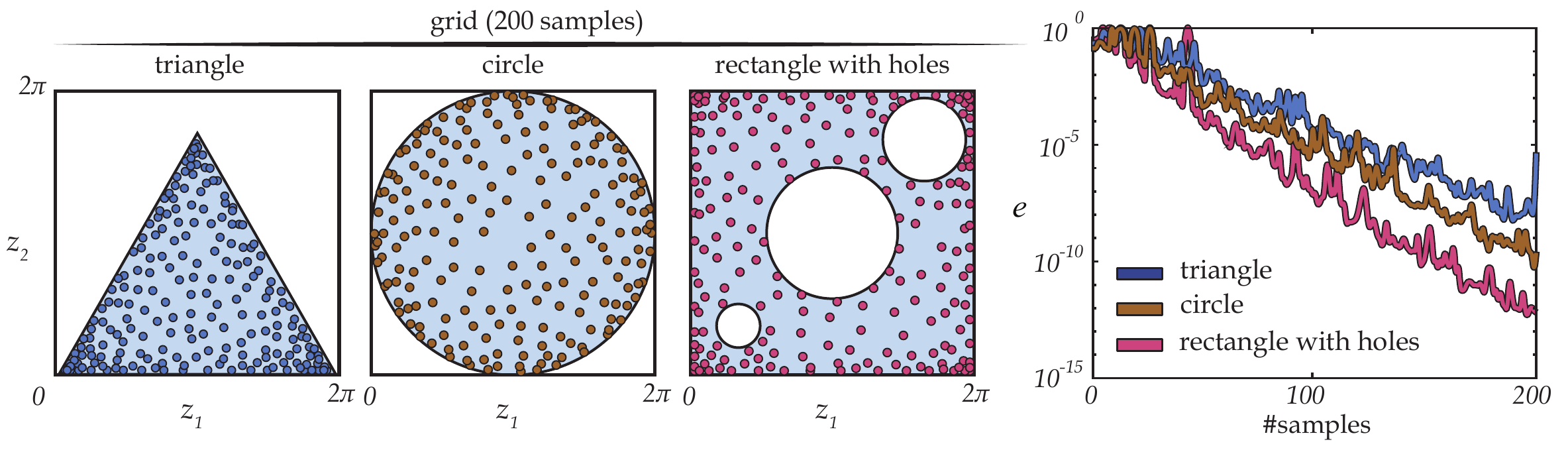}
\caption{\label{fig:AdvDiffComplexDomain} Convergence of surrogate for three different non-hypercube domains.}
\end{figure*}\\
The results show an exponential convergence behaviour for all three different geometries, but the convergence rates slightly differ, which is likely caused by the choice of basis, which is suboptimal for all domains. Furthermore, weighing the residual with the PDF in case of a non-uniformly distributed uncertainties is again possible, but results are similar to the results shown in figure \ref{fig:AdvDiffPDF} and are not shown for repetitiveness.

It has to be pointed out that the choice of a lexicographically ordered Chebyshev basis is far from optimal on non-hypercube domains. Nevertheless, good convergence behaviour is still achieved.

\subsection*{Shallow water equations with dependent random inputs}~\\
\noindent As last test-case, we study the performance of the NIPPAS method for a hyperbolic non-linear PDE with dependent random inputs in a non-hypercube random space. In this test-case the underlying model is non-linear and comprises three uncertain parameters which lie on a 2D triangular manifold in a 3D space and therefore combines the difficult aspects from all previous test-cases. The underlying model is a system of conservation laws, namely the 1D shallow water equations (SWEs). The SWEs describe inviscid fluid flow with a free surface, under the action of gravity, with the thickness of the fluid layer small compared to the other length scales \cite{vreugdenhil_numerical_2013}:
\begin{equation}
\der{}{t}\vecbrace{h \\hv} + \der{}{x}\vecbrace{hv \\hv^2 + gh^2/2} = \mbf{0}\ ,
\label{eq:SWE1D}
\end{equation}
where $h$ is the free surface height (thickness of the fluid layer), $v$ the velocity, and $g$ the acceleration of gravity. Reflective boundary conditions are imposed at $x=\pm 1$ and the initial condition for the system of PDEs is given by a Riemann problem:
\begin{equation}
\vecbrace{h \\ v}(x, t=0) = \left\lbrace \begin{array}{ll}
\vecbrace{h_l \\ v_l},& x\leq 0\ ,\\
\vecbrace{1 \\ 0},& x> 0\ ,\\
\end{array} \right. 
\end{equation}
leading to a so-called dambreak problem. The three uncertain inputs are $z_1=g,\ z_2=h_l$ and $z_3=v_l$, and are assumed to jointly follow a Dirichlet distribution, which is the multivariate generalisation of the 1D $\beta$-distribution \cite{hazewinkel_encyclopaedia_1990}, and is often used as a prior to a discrete categorical distribution in Bayesian statistics. The PDF of the $3D$ Dirichlet distribution with shape parameters $(\alpha_1, \alpha_2, \alpha_3)$ is given by:
\begin{equation}
\rho(z_1, z_2, z_3; \alpha_1, \alpha_2, \alpha_3) = \frac{\Gamma(\sum_{i=1}^3 \alpha_i)}{\prod_{i=1}^3 \Gamma(\alpha_i)}\prod_{i=1}^3 z_i^{\alpha_i - 1}\ ,
\label{eq:DirichletPDF}
\end{equation}
which is defined on the unit simplex
\begin{equation}
\sum_{i=1}^3 z_i = 1,\ \text{and}\ z_i\geq 0\ \forall i\ .
\end{equation}
For this specific test-case we scale/translate the unit simplex to a triangle with corner points 
\begin{equation}
(g, h_l, v_l) = \{(12, 0.5, -0.5), (8, 1.5, -0.5), (8, 0.5, 0.5)\}\ .
\end{equation}
For testing the efficiency of refinement measure \eqref{eq:ResidualPDF}, we consider the shape parameter set $(\alpha_1, \alpha_2, \alpha_3) = (5, 2, 2)$. This specific shape parameter set corresponds to an asymmetric PDF and is shown in figure \ref{fig:SWEPDF}.
\begin{figure}[h]
\centering
\includegraphics[width = 0.5\textwidth]{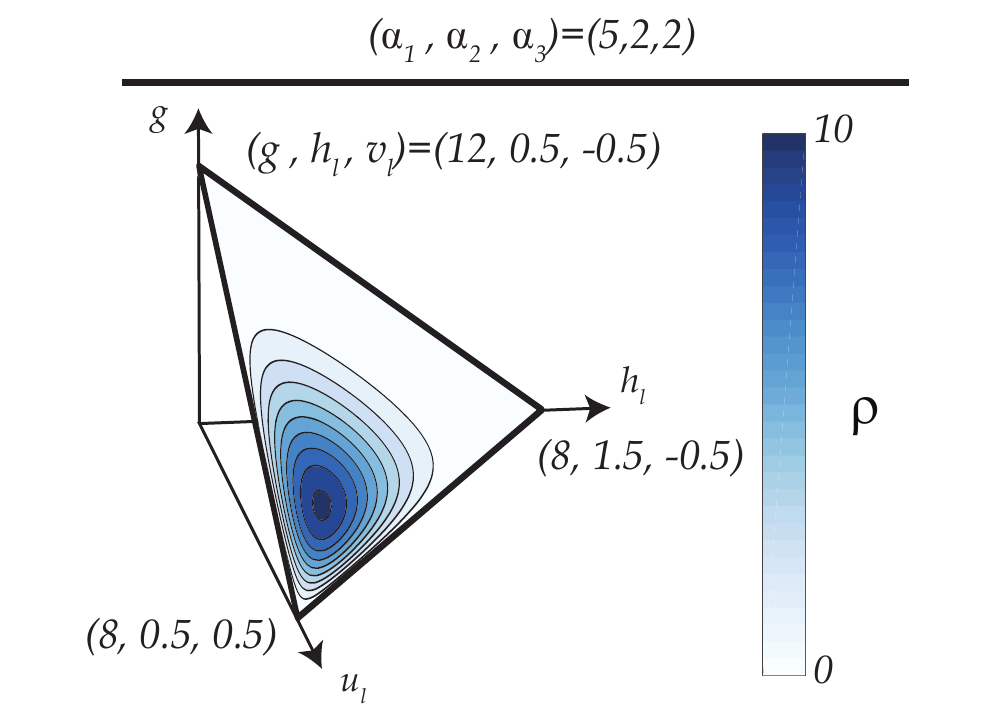}
\caption{\label{fig:SWEPDF} Dirichlet PDF on the scaled unit simplex for $(\alpha_1, \alpha_2, \alpha_3) = (5, 2, 2)$.}
\end{figure}\\
The Dirichlet distribution is defined on a 2D triangular manifold in 3D space and is used for testing the NIPPAS efficiency for non-hypercube random spaces. The interpolation basis is the set of Chebyshev polynomials defined on the smallest hypercube that comprises the 2D manifold. As mentioned before, improvements on the interpolation basis are possible, but are outside the scope of this paper. Furthermore, the QoI $u(z_1, z_2, z_3)$ is the free surface height $h$ at the left boundary $x=-1$ at $t=1$.

The NIPPAS method samples the solution of the dambreak problem for multiple inputs in order to construct a surrogate. A commonly used method for solving the SWEs is a Riemann solver based finite-volume discretisation \cite{leveque_finite_2002}, which can determine accurate solutions efficiently. In this paper instead we demonstrate the effectiveness of the NIPPAS method in combination with a more recently developed numerical method. Instead of using a Riemann solver, a neural network is used to solve the SWEs \cite{raissi_physics_2017}. An advantage of using neural networks for solving PDEs is that \textit{the solution is given in terms of a functional form, from which derivatives can be directly computed analytically}. This functional form allows us to calculate the residual, without requiring alterations to the code output. A multi-output neural network with 7 hidden layers and 40 neurons in each hidden layer is used to solve the SWEs, which is trained on a total of $10^5$ collocation nodes in space and time, which are the places where the neural network tries to enforce the PDE. This particular combination of number of hidden layers and neurons showed the best results and is therefore used in this paper. The trained neural network produces a solution of the PDE. After the residual is computed and a new sample location in random space is determined, the neural network is retrained to produce a solution of the PDE for this new set of parameter values. Previously trained neural networks closest to the new sample location are used as initial starting point for training the new neural network in order to significantly speed-up the training process.
\begin{figure*}[!hbt]
\centering
\includegraphics[width = \textwidth]{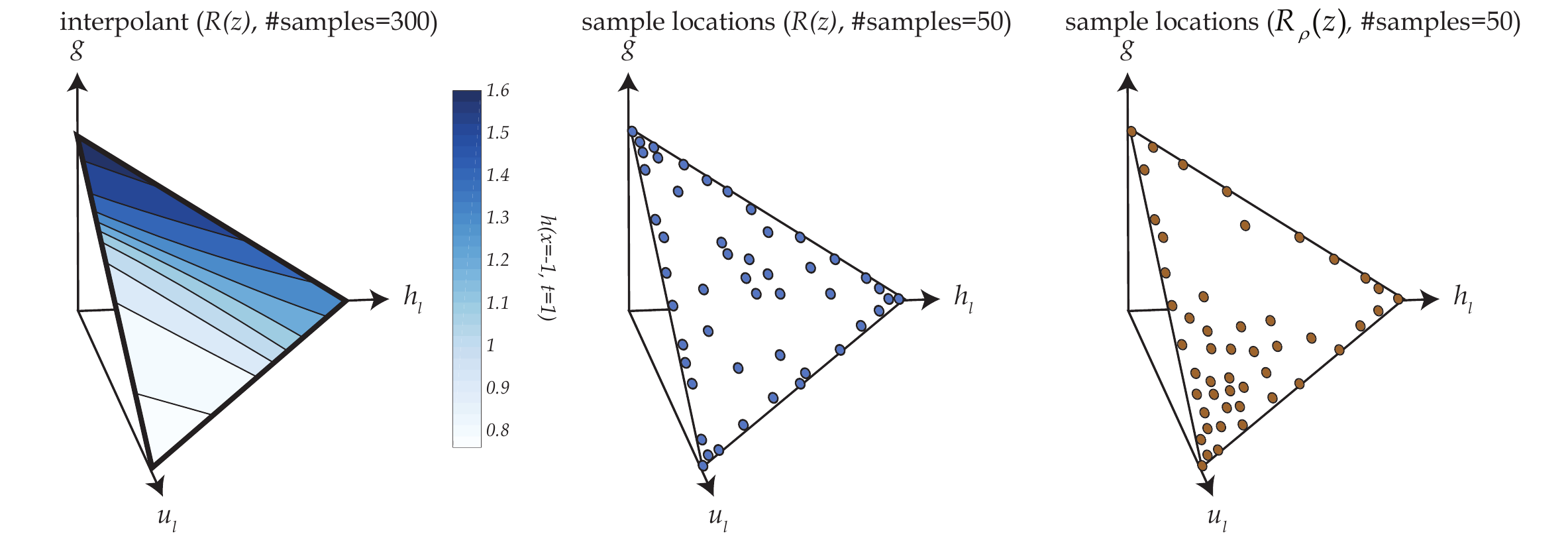}
\caption{\label{fig:SWEResults} (left) Surrogate model based on 300 samples for refinement measure \eqref{eq:Residual}. (right) Sample locations for refinement measures \eqref{eq:Residual} and \eqref{eq:ResidualPDF}, respectively.}
\end{figure*}\\
The surrogate and sample locations for both refinement measures \eqref{eq:Residual} and \eqref{eq:ResidualPDF} are shown in figure \ref{fig:SWEResults}. 
The sample locations show clustering at the boundaries, to produce a stable interpolant. As mentioned before, if the surrogate tends to become unstable and grows at the boundaries of the random space, the residual becomes large at the boundaries as well and causes refinement of the surrogate near the boundaries. However, at early stages of the refinement process the surrogate can still show irregular oscillations, which is due to insufficient refinement at the boundaries, but these disappear upon further refinement. When taking the PDF into account, clustering also occurs in the region of high probability, as expected. This clustering deteriorates the accuracy of the surrogate in regions of low probabilities, but leads to improved estimation of statistical quantities. Convergence comparison for refinement measures \eqref{eq:Residual} and \eqref{eq:ResidualPDF} are shown in figure \ref{fig:SWEResults2}.
\begin{figure*}[!h]
\centering
\includegraphics[width = \textwidth]{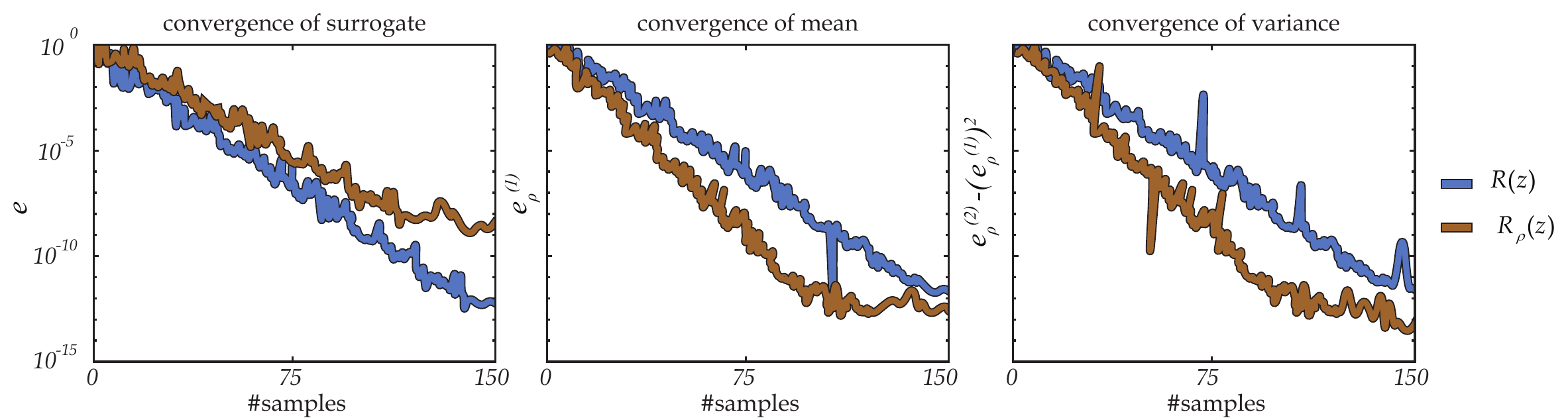}
\caption{\label{fig:SWEResults2} Results for the dambreak problem with random inputs. The convergence plot shows a comparison of the two different refinement measure \eqref{eq:Residual} and \eqref{eq:ResidualPDF}. The error in the surrogate is computed  with \eqref{eq:L2Error} with 5000 Monte-Carlo samples. The errors in the mean and variance are calculated with \eqref{eq:ErrorStatisticalMoments}.}
\end{figure*}\\
The results show indeed faster convergence in statistical quantities when accounting for the PDF in the refinement measure, which was also shown in figure \ref{fig:AdvDiffPDF}.

\section{Conclusion}~\\
\noindent In this paper we have presented a novel approach for parametric surrogate construction when the underlying PDE is known. Our technique, the Non-Intrusive PDE/PDF-informed Adaptive Sampling (NIPPAS), is suited for surrogate construction on non-hypercube parametric spaces. Non-hypercube parametric spaces occur when the underlying PDF is dependent, e.g., Dirichlet-distributed, and significantly complicate surrogate construction when using common stochastic collocation methods, e.g., sparse grid interpolation. The key ingredient of the proposed empirical interpolation procedure is refinement which is based on both the PDE-residual and the PDF. Sampling based on the PDE-residual leads to stable interpolation, even on non-hypercube domains, due to sample clustering at the boundaries of the domain. At the same time, the incorporation of the PDF in the refinement procedure samples the parametric space in regions of high probability, which ensures fast convergence of statistical quantities. This combination makes
the NIPPAS method an efficient and flexible method that is applicable to a wide range of UQ problems.

The NIPPAS method has been applied to several numerical examples: 1D and 2D surrogate construction on a hypercube with linear underlying PDE, 2D surrogate construction on complex domains, and 3D surrogate construction with non-linear underlying PDE on a complex domain. In all cases, exponential convergence is obtained, leading to an accurate surrogate model fast. This surrogate model can be directly used as a tool for uncertainty quantification (for example with Monte-Carlo type methods), but it is also a great tool for the parametric solution of black-box models.

Currently, the interpolation basis for non-hypercube domains is a Chebyshev basis defined on the smallest hypercube that comprises the parametric space. Several improvements could be made, for instance by constructing a suitable basis based on the sample locations \cite{narayan_stochastic_2012}. Furthermore, the global minimisation problem to be solved at each iteration does not scale well to high-dimensional random space, and alternatives for the particle swarm optimisation may be used \cite{goodfellow_deep_2016}.

\section*{Acknowledgements}~\\
\noindent This work is part of the research programme ''SLING'' (Sloshing of Liquefied Natural Gas), which is (partly) financed by the Netherlands Organisation for Scientific Research (NWO).

\section*{References}
\bibliography{MyLibrary}
\end{document}